\documentclass{amsart}
\usepackage{amsmath,amssymb,amsthm,amscd,mathtools}
\usepackage{pgfplots}
\usepackage{mathrsfs}
\newtheorem{formula}{}[section]
\newtheorem{proposition}[formula]{Proposition}
\newtheorem{corollary}[formula]{Corollary}
\newtheorem{lemma}[formula]{Lemma}
\newtheorem{theorem}[formula]{Theorem}
\theoremstyle{definition}
\newtheorem{definition}[formula]{Definition}
\newtheorem{example}[formula]{Example}
\theoremstyle{remark}
\newtheorem*{remark}{Remark}

\begin{document}

\title{Lie algebras of slow growth and Klein-Gordon equation}
\author{Dmitry V. Millionshchikov}
\thanks{This work is supported by RFBR under the grant 17-01-00671}
\address{Department of Mechanics and Mathematics, Moscow
State University, 1 Leninskie gory, 119992 Moscow, Russia}
\email{million@mech.math.msu.su}

\begin{abstract}
We discuss the notion of characteristic Lie algebra of a hyperbolic PDE. The integrability 
of a hyperbolic PDE is closely related to the properties of the corresponding characteristic Lie algebra $\chi$.
We establish two explicit isomorphisms: 

1) the first one is between the characteristic Lie algebra $\chi(\sinh{u})$ of the sinh-Gordon equation $u_{xy}=\sinh{u}$ and the non-negative part  ${\mathcal L}({\mathfrak sl}(2,{\mathbb C}))^{\ge 0}$ of the loop algebra of ${\mathfrak sl}(2,{\mathbb C})$ that corresponds to the Kac-Moody algebra $A_1^{(1)}$
$$
\chi(\sinh{u})\cong {\mathcal L}({\mathfrak sl}(2,{\mathbb C}))^{\ge 0}={\mathfrak sl}(2, {\mathbb C}) \otimes {\mathbb C}[t].
$$

2) the second isomorphism is for the Tzitzeica equation $u_{xy}=e^u{+}e^{-2u}$
$$
\chi(e^u{+}e^{-2u}) \cong {\mathcal L}({\mathfrak sl}(3,{\mathbb C}), \mu)^{\ge0}=\bigoplus_{j=0}^{+\infty}{\mathfrak g}_{j ({\rm mod} \; 2)} \otimes t^j, 
$$ 
where ${\mathcal L}({\mathfrak sl}(3,{\mathbb C}), \mu)=\bigoplus_{j \in {\mathbb Z}}{\mathfrak g}_{j ({\rm mod} \; 2)} \otimes t^j$ is the twisted loop algebra of the simple Lie algebra  ${\mathfrak sl}(3,{\mathbb C})$ that corresponds to the Kac-Moody algebra $A_2^{(2)}$.

Hence the Lie algebras $\chi(\sinh{u})$ and  $\chi(e^u{+}e^{-2u})$ are slowly linearly growing Lie algebras with average growth rates $\frac{3}{2}$ and $\frac{4}{3}$ respectively.

\end{abstract}
\date{}

\maketitle

\section*{Introduction}

The concept of characteristic Lie algebra $\chi(f)$ of a hyperbolic system of PDE
\begin{equation}
\label{KG_System}
u^i_{xy}=f^i(u^1,\dots,u^n), i=1,\dots, n,
\end{equation}
 was  introduced by Leznov, Smirnov, Shabat and Yamilov \cite{ShYa, LSmSh}. It is a natural generalization of the notion of characteristic vector field of a hyperbolic PDE that was first proposed by Goursat in 1899. In his classical paper \cite{G}  Goursat introduced a very effective algebraic approach to the problem of classifying Darboux-integrable equations.

In spite of the rather large number of papers where this algebraic object is studied \cite{ShYa, LSmSh, ZMHSh, ZM, Sakieva}, it can not be said that there exists any completely unambiguous definition of characteristic Lie algebra $\chi$ of a hyperbolic non-linear PDE. We use the definition of characteristic Lie algebra proposed in the initial papers \cite{ShYa, LSmSh}.

An important step in the study of hyperbolic nonlinear Liouville-type systems was made in \cite{Leznov, LSmSh, LSS} where so-called exponential hyperbolic systems were considered
$$
u_{x y}^j=e^{\rho_j}, \; \rho_j=a_{j1}u^1+\dots+a_{jn}u^n, \;j=1,\dots, n.
$$

It pas proved in \cite{Leznov} that if $A=(a_{ij})$ is a non-degenerate Cartan matrix the exponential hyperbolic system (\ref{exp_Cartan}) is Darboux-integrable. The proof \cite{Leznov}  consisted in the construction of a complete solution in
an explicit form which depends on $2n$ arbitrary functions, thus generalizing the one-dimensional case of the classical Liuoville equation $u_{xy}=e^u$.
Later it was claimed in the preprint \cite{ShYa} that the  main  result in \cite{Leznov} can be extended to an arbitrary  generalized Cartan matrix $A$ (possibly degenerate) if we apply the inverse scattering problem method. The two-dimensional case $n=2$  was studied explicitly in  \cite{ShYa, LSmSh}.
\begin{equation}
\label{two_dim}
\left \{ \begin{array}{c}
u^1_{x y}=e^{(a_{11}u^1+a_{12}u^2)},\\
u^2_{x y}=e^{(a_{21}u^1+a_{22}u^2)},\\
\end{array}\right.,
\; A=\begin{pmatrix}
a_{11} & a_{12}\\
a_{21} & a_{22}
\end{pmatrix}.
\end{equation}

It was proved  in \cite {ShYa, LSmSh} that 
for the generalized  Cartan matrices 
$$
A_1=\begin{pmatrix} 2 & {-}2 \\ {-}2 & 2\end{pmatrix}, \; 
A_2=\begin{pmatrix} 2 & {-}4 \\ {-}1 & 2\end{pmatrix}
$$
the corresponding exponential systems (\ref{two_dim}) are integrable by the inverse scattering method. Moreover, the commutants $[\chi(A_1),\chi(A_1)], [\chi(A_2), \chi(A_2)]$ of the corresponding characteristic Lie algebras 
$\chi(A_1),\chi(A_2)$ are isomorphic  to maximal pro-nilpotent subalgebras $N(A^{(1)}_1), N(A^{(2)}_2)$ of the Kac-Moody algebras $A^{(1)}_1, A^{(2)}_2$ respectively (that correspond to the generalized Cartan matrices $A_1$ and $A_2$). Exponential systems  (\ref{two_dim}) corresponding to nondegenerate Cartan $2\times2$-matrices 
$$
\begin{pmatrix} 2 & 0 \\ 0 & 2\end{pmatrix}, 
\begin{pmatrix} 2 & {-}1 \\ {-}1 & 2\end{pmatrix}
\begin{pmatrix} 2 & {-}2 \\ {-}1 & 2\end{pmatrix}, 
\begin{pmatrix} 2 & {-}3 \\ {-}1 & 2\end{pmatrix}
$$
of semisimple Lie algebras $A_1{\oplus}A_1, A_2, C_2, G_2$ are Darboux-integrable. Their characteristic Lie algebras are finite-dimensional solvable Lie subalgebras in semisimple Lie algebras listed above. These finite-dimensional solvable Lie algebras and infinite-dimensional characteristic Lie algebras $\chi(A_1),\chi(A_2)$  were unified into a class of slowly growing Lie algebras \cite{ShYa, LSmSh}.

Originally \cite{LSmSh, ShYa}, when it was talked about the characteristic Lie algebra of finite growth, it was in mind Kac's classification \cite{Kac1} of simple ${\mathbb Z}$-graded Lie algebras of finite growth. The condition of simple ${\mathbb Z}$-grading is very restrictive, meanwhile, the growth of a characteristic Lie algebra $\chi(f)$  must be understood from the point of view of the behavior of its growth function $F_{\mathfrak g}(n)$, i.e. the asymptotics of the dimension $F_{\mathfrak g}(n)=\dim V_n$ of the space $V_n$ of commutators of order at most $n$ of generators.

A finitely generated characteristic Lie algebra $\chi(f)$ of a hyperbolic Klein-Gordon system (\ref{KG_System}) is a pro-solvable Lie algebra whose  commutant $[\chi(f),\chi(f)]$ is a pro-nilpotent naturally graded Lie algebra.

By {\bf Lemma \ref{growth_functions}} we assert that the growth functions of $\chi(f)$ and its commutant $[\chi(f),\chi(f)]$ differ by a positive constant $C(\chi(f))$, which  equals to the dimension of the maximal toral subalgebra of $\chi(f)$.
$$
F_{\chi(f)}(n)=F_{[\chi(f),\chi(f)]}(n)+C(\chi(f)).
$$
Thus, the study of the growth function $F_{\chi(f)}(n)$ of the entire characteristic Lie algebra $\chi(f)$ reduces to studying the growth of the commutant $[\chi(f),\chi(f)]$.

The problem of classification of  ${\mathbb N}$-graded Lie algebras of slow growth is much more complicated problem than the classification of simple ${\mathbb Z}$-graded Lie algebras of finite growth. The Kac list \cite{Kac1} contains a countable number of different Lie algebras,  meanwhile in the case of naturally graded Lie algebras with two generators, an uncountable family of pairwise non-isomorphic Lie algebras of linear growth appears \cite{Mill}. 
There are only three Klein-Gordon equations admitting non-trivial higher symmetries \cite{ZhSh1}.
\begin{itemize}
\item Liouville equation $u_{xy}=e^u$;
\item sinh-Gordon equation
$u_{xy}=\sinh{u}$;
\item Tzitzeica equation
$u_{xy}=e^u+e^{-2u}$.
\end{itemize}

1) It's an elementary exercise to show that
the characteristic Lie algebra   $\chi(e^u)$ of the Liuoville equation 
is the two-dimensional solvable Lie algebra. It can be defined by its basis $X_0, X_1$ and the  unique relation $[X_0, X_1]=X_1$. Its commutant 
$[\chi(e^u),\chi(e^u)]$ is one-dimensional abelian Lie algebra spanned by  $X_1$.

We study two remaining cases and prove

2) {\bf Theorem \ref{sinh-Gordon_theorem}}. {\it The characteristic Lie algebra $\chi(\sinh{u})$ of the sinh-Gordon equation 
$u_{xy}=\sinh{u}$
is  isomorphic to the polynomial loop algebra ${\mathcal L}({\mathfrak sl}(2,{\mathbb C}))^{\ge0}$
$$
\chi(\sinh{u})\cong {\mathcal L}({\mathfrak sl}(2,{\mathbb C}))^{\ge0}={\mathfrak sl}(2,{\mathbb C}) {\otimes} {\mathbb C}[t],
$$  
Its commutant $[\chi(\sinh{u}),\chi(\sinh{u})]$ is isomorphic to the maximal pro-nilpotent Lie subalgebra $N(A_1^{(1)})$
of the Kac-Moody algebra $A_1^{(1)}$.}

3)  {\bf Theorem \ref{Tzitzeica_theorem}}. {\it The characteristic Lie algebra $\chi(e^u{+}e^{-2u})$ of the Tzitzeica equation 
$u_{xy}=e^u{+}e^{-2u}$
is  isomorphic to the twisted polynomial loop algebra ${\mathcal L}({\mathfrak sl}(3,{\mathbb C}), \mu)^{\ge0}$
$$
\chi(e^u{+}e^{-2u})\cong {\mathcal L}({\mathfrak sl}(3,{\mathbb C}), \mu)^{\ge0}=\bigoplus_{j =0}^{+\infty}{\mathfrak g}_{j ({\rm mod} \; 2)} \otimes t^j, {\mathfrak sl}(3,{\mathbb C})={\mathfrak g}_0\oplus{\mathfrak g}_1
$$  
where $\mu$ is a diagram automorphism of ${\mathfrak sl}(3,{\mathbb C})$, $\mu^2={\rm Id}$, and ${\mathfrak g}_0$,  ${\mathfrak g}_1$ are eigen-spaces of $\mu$ corresponding to eigen-values $1,{-}1$ respectively,  $[{\mathfrak g}_s,{\mathfrak g}_q] \subset {\mathfrak g}_{s+q \;({\rm mod} \;2)}$.

Its commutant $[\chi(e^u{+}e^{-2u}),\chi(e^u{+}e^{-2u})]$ is isomorphic to the maximal pro-nilpotent Lie subalgebra $N(A_2^{(2)})$
of the Kac-Moody algebra $A_2^{(2)}$.}

At this point some very important observations need to be made. 

It was discussed in  \cite{LSmSh, LS}  that there is a reduction of two-dimensional systems  (\ref{two_dim}) with matrices $A_1$ and $A_2$ to the sine-Gordon and Tzitzeika equations respectively. However,  explicitly the characteristic Lie algebras $\chi(\sinh{u})$ and  $\chi(e^u{+}e^{-2u})$ have not been calculated there. The question of describing such algebras is very important, because, the characteristic Lie algebras of the one-dimensional and two-dimensional systems (\ref{two_dim}) are different by the definition. This circumstance, as well as some gaps in proofs of  \cite{LSmSh, LS} led to the appearance of \cite{ZM, Sakieva}, where the problem of an explicit description of characteristic Lie algebras $\chi(\sinh{u})$ and  $\chi(e^u{+}e^{-2u})$ was posed and solved. It was solved from the point of view of constructing infinite bases and structure relations (different from the bases and relations proposed in this article). However   
the extremely important  relationship between the characteristic Lie algebras of  $\chi(\sinh{u})$ and  $\chi(e^u{+}e^{-2u})$ of the sinh-Gordon and Tzitzeica equations and affine Kac-Moody algebras $A_1^{(1)}$ and $A_2^{(2)}$ escaped the attention of the authors in \cite{ZM, Sakieva}. In addition, we wrote the generators of these algebras in terms of Bell polynomials, which helped us to determine and relate various gradings of $\chi(\sinh{u})$ and  $\chi(e^u{+}e^{-2u})$. Also an interesting feature was the observation that the Lie algebras $\chi(\sinh{u})$ and $\chi(\sin{u})$ are non-isomorphic over ${\mathbb R}$ (but isomorphic over  ${\mathbb C}$).

The author is grateful to Sergey Smirnov and Victor Buchstaber for valuable comments and remarks.

\section{Characteristic Lie algebra of hyperbolic non-linear PDE}
Here and in the sequel, we define, if not specifically stated, all Lie algebras over the field ${\mathbb K}$, which is either the field ${\mathbb R}$ of reals or the field ${\mathbb C}$ of complex numbers.

Consider a  system of hyperbolic PDE
\begin{equation}
\label{hyperbolic_s}
u_{x y}^j=f^j(u), \;j=1,\dots, n, \; u=(u^1,\dots,u^n),
\end{equation}
where  each function $f^j(u), j=1,\dots,n,$ belongs to a ${\mathbb K}$-algebra $C^{\omega}(\Omega)$  of (locally) analytic ${\mathbb K}$-valued functions of $n$ real variables $u=(u^1,\dots,u^n)$ defined on some open domain $\Omega \subset {\mathbb R}^n$ (it is more convinient to consider germs instead of functions, but we will keep the definition from \cite{LSmSh, ShYa}). By $x,y$ we denote two coordinates on the real plane ${\mathbb R}^2$ and assume solutions of (\ref{hyperbolic_s}) to be (locally) analytic functions of $x,y$. 

Take an algebra $C^{\omega}(\Omega)[u_1,u_2,\dots]=C^{\omega}(\Omega)[u_1^1,\dots,u_1^n,u_2^1,\dots, u_2^n,\dots]$ of polynomials in an infinite set of variables $\{ u_i=(u^1_i,\dots,u^n_i), i \ge 1\}$ with coefficients in $C^{\omega}(\Omega)$. The multiplicative structure in $C^{\omega}(\Omega)[u_1,u_2,\dots]$ is defined as the standard 
product of polynomials.
\begin{example}For $n=2$ the following polynomial 
$$
P(u^1,u^2;u_1^1,u_1^2,u_2^1,u_2^2,\dots)=\sin{(u^1{+}2u^2)}\cdot (u_1^1)^2+2\cos{u^1}\cdot (u_2^1)^3,
$$
belongs to the algebra $C^{\omega}(\Omega)[u_1^1,u^2_1,u_2^1,u_2^2,\dots]=C^{\omega}(\Omega)[u_1,u_2,\dots]$.
\end{example}

Define a Lie algebra ${\mathcal L}$ of the first order linear differential operators of the form
\begin{equation}
\label{vector_fields}
X=\sum_{k=1}^{+\infty}P_k^{\alpha}(u;u_1,u_2,\dots)\frac{\partial}{\partial u^{\alpha}_k},
\end{equation}
where all coefficients $P_i^{\alpha}(u;u_1,u_2,\dots), \alpha=1,\dots,n, i \ge 1$, are polynomials from   $C^{\omega}(\Omega)[u_1,u_2,\dots]$. We used tensor rules in (\ref{vector_fields}) for summation
$$
P_i^{\alpha}(u;u_1,u_2,\dots) \frac{\partial}{\partial u^{\alpha}_k}=\sum_{\alpha=1}^n P_i^{\alpha}(u;u_1,u_2,\dots)\frac{\partial}{\partial u^{\alpha}_k}.
$$
\begin{proposition}
One can remark that ${\mathcal L}$ is an example of  $({\mathbb K}, C^{\omega}(\Omega))$-Lie algebra (Lie-Reinhart algebra) \cite{R} with trivial action of ${\mathcal L}$ on $C^{\omega}(\Omega))$.
$$
[f(u)X,g(u)Y]=f(u)g(u)[X,Y], \; f(u), g(u) \in C^{\omega}(\Omega), \; X, Y \in 
{\mathcal L}.
$$
\end{proposition}
 \begin{remark}
We have already said in the Introduction that there does not seem to exist a canonical definition of the characteristic Lie algebra of a hyperbolic PDE.
To all appearances, the characteristic Lie algebra $\chi(f)$ of a hyperbolic equation $u_{xy}=f(u)$ with additional structure of a $({\mathbb K}, C^{\omega}(\Omega)$-Lie algebra with trivial $\chi(f)$-action on $C^{\omega}(\Omega)$ is called the characteristic Lie ring of $u_{xy}=f(u)$ in a series of papers \cite{ZMHSh, ZM, Sakieva} et al. Linear dependence or independence of vector fields, the choice of basis in the characteristic Lie algebra $\chi(f)$ is understood in \cite{ZMHSh, ZM, Sakieva} with respect to the  left module structure over the localization of $C^{\omega}(\Omega)$.

In \cite{ShYa}, one of the very first and key papers on the characteristic Lie algebras  of hyperbolic systems of PDE,  vector fields are considered for some fixed value $u_M$ of the variables $u=(u^1,\dots,u^n)$. 

More precisely, let $M=(u^1_M,\dots, u^n_M)=u_M$ be a fixed  point in $\Omega$. One can consider an evaluation map
$ev: {\mathcal L} \to {\mathcal L}$ defined by
$$
X=\sum_{k=1}^{+\infty}P_k^{\alpha}(u;u_1,u_2,\dots)\frac{\partial}{\partial u^{\alpha}_k}
 \xrightarrow{\text{ev}_M}
X_{M}= 
\sum_{k=1}^{+\infty}P_k^{\alpha}(u_M;u_1,u_2,\dots)\frac{\partial}{\partial u^{\alpha}_k}.
$$
Sometimes by characteristic Lie algebra $\chi(f)$ of a hyperbolic equation $u_{xy}=f(u)$
is called the image $ev_M(\chi(f))$ of the evaluation map $ev_M$ for some choice  of a point $M \in \Omega$ \cite{ShYa}. Thus, the Lie algebra $ev_M(\chi(f))$ consists of first order linear differential operators $\sum_{k=1}^{+\infty}P_k\frac{\partial}{\partial u_k}$ with coefficients 
$P_k$ taken from the standard polynomial ring ${\mathbb C}[u_1,u_2,\dots]$ \cite{ShYa}. 
\end{remark}

Consider commuting operators $\frac{\partial}{\partial u^j}, j=1,\dots, n$. The following formulas are valid
$$
\left[\frac{\partial}{\partial u^j}, X\right]=\left[\frac{\partial}{\partial u^j},\sum_{k=1}^{+\infty}P_k^{\alpha}(u;u_1,u_2,\dots)\frac{\partial}{\partial u^{\alpha}_k}\right]=
\sum_{k=1}^{+\infty}\frac{\partial P_k^{\alpha}(u;u_1,u_2,\dots)}{\partial u^j}\frac{\partial}{\partial u^{\alpha}_k}
$$

Consider an operator $D: C^{\omega}(\Omega)[u_1,u_2,\dots] \to C^{\omega}(\Omega)(\Omega)[u_1,u_2,\dots]$. 
\begin{equation}
\label{operator_D}
D=u^{\alpha}_1\frac{\partial}{\partial u^{\alpha}}+u^{\alpha}_2\frac{\partial}{\partial u_1^{\alpha}}+u^{\alpha}_3\frac{\partial}{\partial u_2^{\alpha}}+\dots+u^{\alpha}_{k+1}\frac{\partial}{\partial u_k^{\alpha}}+\dots,
\end{equation}

The operator $D$ is called the operator of the full partial derivative $\frac{\partial }{\partial x}$. 
The definition of the operator $D$ has a formal algebraic meaning, but the formula (\ref{operator_D}) defining it has a completely concrete analytic origin.
 Indeed, consider a solution $u(x,y)=(u^1(x,y),\dots, u^n(x,y))$  of the system (\ref{hyperbolic_s}). Let $g^j(u;u_1,u_2,\dots) \in C^{\omega}(\Omega)[u_1,u_2,\dots], j=1,\dots,n$.  Define with a help of  $u(x,y)$  a composite function $g (x,y)=(g^1(x,y),\dots,g^n(x,y))$ of two arguments $x,y$:
$$
g^j(x,y)=g^j(u(x,y); u^1_x(x,y),\dots,u^n_x(x,y),u^1_{xx}(x,y),\dots,u^n_{xx}(x,y),\dots)
$$
In other words, we have a parametrization $u^{\alpha}_j=\frac{\partial^j u^{\alpha}}{\partial x^j}, \alpha=1,\dots,n, j \ge 1$.
$$
(u^1_1,\dots,u^n_1)=(u^1_x,\dots,u^n_x), (u^1_2,\dots,u^n_2)=(u^1_{xx},\dots,u^n_{xx}), \dots, 
$$
In particular we have obvious formulas
$$
\frac{\partial u^{\alpha}}{\partial x}=D(u^{\alpha})=u_{1}^{\alpha}, \frac{\partial u_k^{\alpha}}{\partial x}=D(u_k^{\alpha})=u_{k+1}^{\alpha}, \; \alpha=1,\dots,n, k \ge 1.
$$
Computing the partial derivative $\frac{\partial  g^j}{\partial x}$ of the composite function $g^j(x,y)$, we obtain
$$
\frac{\partial  g^j}{\partial x}=\frac{\partial u^{\alpha}}{\partial x}\frac{\partial  g^j}{\partial u^{\alpha}}+\frac{\partial u_1^{\alpha}}{\partial x}\frac{\partial  g^j}{\partial u_1^{\alpha}}+\dots+\frac{\partial u_k^{\alpha}}{\partial x}\frac{\partial g^j} {\partial u_k^{\alpha}}+\dots=u^{\alpha}_1\frac{\partial g^j}{\partial u^{\alpha}}+u^{\alpha}_2\frac{\partial g^j}{\partial u_1^{\alpha}}+\dots+u^{\alpha}_{k+1}\frac{\partial g^j}{\partial u_k^{\alpha}}+\dots
$$
Similar arguments lead us to the formula for  $\frac{\partial}{\partial y}$ "on solutions of" (\ref{hyperbolic_s})
$$
X(f)=\frac{\partial}{\partial y}=f^{\alpha} \frac{\partial}{\partial u^{\alpha}_1}+D(f^{\alpha}) \frac{\partial}{\partial u^{\alpha}_2}+D^2(f^{\alpha}) \frac{\partial}{\partial u^{\alpha}_3}+\dots+D^{k+1}(f^{\alpha}) \frac{\partial}{\partial u^{\alpha}_k}+\dots
$$
\begin{definition}[\cite{LSmSh, ShYa}]
\label{Char_Lie}
A Lie algebra $\chi(f)$ generated by $n+1$ vector fields
$$
X(f), \frac{\partial}{\partial u^1}, \dots, \frac{\partial}{\partial u^n},
$$
is called characteristic Lie algebra   of  the hyperbolic system (\ref{hyperbolic_s}).
\end{definition}
A linear span of $\frac{\partial}{\partial u^1}, \dots, \frac{\partial}{\partial u^n}$ determines an abelian subalgebra
$\chi_0(f)$ of $\chi(f)$. One can easily verify the followng commutation relations 
$\frac{\partial}{\partial u^j}$ with $X(f)$
$$
\left[ \frac{\partial}{\partial u^j}, X(f)\right]=X\left( \frac{\partial f}{\partial u^j}\right)=\sum_{k=1}^{+\infty} D^{k+1}\left( \frac{\partial f^{\alpha}}{\partial u^j}\right) \frac{\partial }{\partial u^{\alpha}_k}, j=1,\dots, n.
$$
We denote by $\chi_1(f)$ the smallest invariant subspace of $\chi_0$-action on  $\chi(f)$  containing  the operator $X(f)$. The subspace $\chi_1(f)$ coinsides with  the linear span of all operators $X\left(\frac{\partial^s f}{\partial u^{j_1}\dots\partial u^{j_s}}\right), s \ge 0$ and we have
$$
\left[ \chi_0(f), \chi_1(f)\right] = \chi_1(f). 
$$

In this article we are interested mainly in the one-dimensional case $n=1$.  The corresponding scalar PDE is well known and sometimes it is called Klein-Gordon equation \cite{ZhSh1, ZMHSh}.

Indeed, consider a classical Klein-Gordon equation
$$
u_{tt}-u_{zz}=f(u).
$$ 
Making a linear change of variables $x=\frac{z+t}{2}, y=\frac{z-t}{2}$ we'll
get 
\begin{equation}
\label{klein-gordon}
u_{xy}=f(u),
\end{equation}
where we assume that $f(u)$ is a locally analytic function on one variable $u$.
Further in the text we will call by the Klein-Gordon equation the equation in the form 

The operator $D$ of the full derivative with respect to $x$ is 
$$
D=u_1\frac{\partial}{\partial u}+u_2\frac{\partial}{\partial u_1}+u_3\frac{\partial}{\partial u_2}+\dots+u_{n+1}\frac{\partial}{\partial u_n}+\dots ,
$$

We recall  that  $u_i$ are parametrized  by means of some solution $u(x,y)$ of (\ref{klein-gordon}).
$$
u_1=u_x, u_2=u_{xx}, \dots, u_i=\frac{\partial^i u}{ \partial x^i}, \dots
$$
In this case we have
$$
\frac{\partial}{\partial x}(g(u(x,y),u_1(x,y),u_2(x,y),\dots)=D(g(u(x,y),u_1(x,y),u_2(x,y),\dots)).
$$
\begin{example}
It's an elementary exercise to verify by recursion that 
$$
D^k(e^{\lambda u})=e^{\lambda u}B_k(\lambda u_1,\dots,\lambda u_k),
$$  
where 
$$
B_k(\lambda u_1,\dots,\lambda u_k)=u_1^k\lambda^k+\dots+u_k\lambda, \quad k=0,1,2, \dots,
$$ 
are complete Bell polynomials of degree $k$. Complete Bell polynomials are well-known combinatorial object and they have a lot of properties and applications, see \cite{A} for references. We want just to recall only a few basic facts about them. 

Complete Bell polynomials can be defined recursively by the formula
$$
B_{n{+}1}(u_1,u_2, \dots, u_{n+1})=\sum_{i=0}^n \binom{n}{i}B_{n{-}i}(u_1,u_2, \dots, u_{n-i})u_{i+1},
$$
with the  initial condition $B_0=1$. 

The first few complete Bell polynomials are:
$$
\begin{array}{l}
B_1(u_1)=u_1, B_2(u_1,u_2)=u_1^2+u_2, 
B_3(u_1,u_2,u_3)=u_1^3+3u_1u_2+u_3,\\
B_4(u_1,u_2,u_3,u_4)=u_1^4+6u_1^2u_2+4u_1u_3+3u_2^2+u_4, \dots
\end{array}
$$
The generating function for complete Bell polynomials is 
$$
\exp{\left( \sum_{i=1}^{+\infty} u_i\frac{t^i}{i!}\right)}=\sum_{n=0}^{+\infty}B_n(u_1,\dots,u_n)\frac{t^n}{n!}.
$$
\end{example}
A one-dimensional version of the general Definition \ref{Char_Lie} is 
\begin{definition}
The characterisitic Lie algebra $\chi(f)$ of Klein-Gordon equation (\ref{klein-gordon}) is a Lie algebra of vector fields generated by two vector fields $X_0$ and $X_1$
$$
X_0=\frac{\partial}{\partial u}, 
X_1=X(f)=f\frac{\partial}{\partial u_1}+D(f)\frac{\partial}{\partial u_2}+\dots+D^{n-1}(f)\frac{\partial}{\partial u_n}+\dots
$$
\end{definition}
It is an elementary exercise to express $D^{k}(f)$ for an arbitrary analytic $f$ in terms of complete differential Bell polynomials $B_n(u_1\frac{d}{du},\dots,u_n\frac{d}{du})$, i.e. 
$$D^n(f)=B_n(u_1\frac{d}{du},\dots,u_n\frac{d}{du})(f), $$
where first four differential Bell polynomials are
$$
\begin{array}{l}
B_0=1, B_1(u_1\frac{d}{du})=u_1\frac{d}{du}, B_2(u_1\frac{d}{du},u_2\frac{d}{du})=u_1^2\frac{d^2}{du^2}+u_2\frac{d}{du}, \\
B_3(u_1\frac{d}{du},u_2\frac{d}{du},u_3\frac{d}{du})=u_1^3\frac{d^3}{du^3}+3u_1u_2\frac{d^2}{du^2}+u_3\frac{d}{du},\\
B_4(u_1\frac{d}{du},u_2\frac{d}{du},u_3\frac{d}{du},u_4\frac{d}{du})=u_1^4\frac{d^4}{du^4}{+}6u_1^2u_2\frac{d^3}{du^3}{+}(4u_1u_3{+}3u_2^2)\frac{d^2}{du^2}{+}u_4\frac{d}{du}.
\end{array}
$$
One can express $D^{n}(f)$ in terms of incomplete (partial) Bell polynomials $B_{n,k}(u_1, \dots,u_{n-k+1})$ (see \cite{A} for references)
$$
D^{n}(f)=\sum_{k=1}^{n}B_{n,k}(u_1, \dots,u_{n-k+1})\frac{d^k f}{dx^k}.
$$
We recall here only the generating function for incomplete Bell polynomials
$$
\exp{\left( z\sum_{i=1}^{+\infty} u_i\frac{t^i}{i!}\right)}=\sum_{n,k\ge 0} B_{n,k}(u_1,\dots,u_{n-k+1})z^k\frac{t^n}{n!}.
$$
and the relationship between the complete and incomplete polynomials
$$
B_{n}(u_1,\dots,u_{n})=\sum_{k= 1}^n B_{n,k}(u_1,\dots,u_{n-k+1})
$$
We denote  by $Y_1$ the commutator
$$
Y_1=[X_0, X_1]=f_u\frac{\partial}{\partial u_1}+D(f_u)\frac{\partial}{\partial u_2}+\dots+D^{n-1}(f_u)\frac{\partial}{\partial u_n}+\dots
$$
We have also
$$
[X_0, Y_1]=\left[X_0,[X_0,X_1] \right]=f_{uu}\frac{\partial}{\partial u_1}+D(f_{uu})\frac{\partial}{\partial u_2}+\dots+D^{n-1}(f_{uu})\frac{\partial}{\partial u_n}+\dots
$$
\begin{example}
Let $f(u)=e^u$. We have the Liouville equation $u_{xy}=e^u$. It follows that  $[X_0, X_1]=X_1$ in this case.
Hence the characteristic Lie algebra   $\chi(e^u)$ of the Liuoville equation 
is the non-abelian two-dimensional solvable Lie algebra. It can be defined by its basis $X_0, X_1$ and the  unique commutation relation
$$
[X_0, X_1]=X_1.
$$
\end{example}

We have already noted that the implication of the canonical definition of the characteristic Lie algebra is related to its auxiliary, albeit very important, role in the search for integrals and higher symmetries of those hyperbolic equations by which they are constructed.

\begin{definition}
A (locally) analytic function $w(u;u_1,\dots,u_n)$ is called $x$-integral of PDE $u_{xy}=f(u)$ if 
\begin{equation}
\label{char_equation}
\frac{\partial }{\partial y}w\left(u,u_x,u_{xx},\dots, \frac{\partial^n u}{\partial x^n}\right)=0,
\end{equation}
where $u(x,y)$ is a solution of $u_{xy}=f(u)$.
\end{definition}
Respectively a (locally) analytic function $w(u_1,\dots,u_n)$ is called $y$-integral of PDE $u_{xy}=f(u)$ if
$$
\frac{\partial }{\partial x}w\left(u, u_y,u_{yy},\dots, \frac{\partial^n u}{\partial y^n}\right)=0, \; u_{xy}=f(u).
$$
Evidently in our symmetric case $u_{xy}=f(u)$ a $x$-integral $w$ defines 
a $y$-integral and vise versa.
The equation (\ref{char_equation}) can be written now as 
$$
u_1\frac{\partial w}{\partial u}+X(f)w=0.
$$
and it is equivalent to the system
$$
\frac{\partial w}{\partial u}=0, X(f)w=0.
$$
In other words a $x$-integral $w$ is annihilated by two generators 
 $\frac{\partial}{\partial u}, X(f)$ of characteristic Lie algebra $\chi(f)$ and hence it is annihilated by the whole Lie algebra  $\chi(f)$.

\begin{example}
A second order polynomial $w_2(u_1, u_{2})=\frac{1}{2}(u_1)^2-u_{2}$ determines both $x$-, $y$-integrals  of the Liouville equation $u_{xy}=e^u$. 
$$
X(e^u)w_2(u_1,u_2)=e^u\left(\frac{\partial}{\partial u_1}+u_1\frac{\partial}{\partial u_2}\right)\left( \frac{1}{2}(u_1)^2-u_{2}\right)=0.
$$
Or one can verify directly that $w_2(u_x, u_{xx})=\frac{1}{2}(u_x)^2-u_{xx}$ is a $x$-integral 
$$
((u_x)^2-2u_{xx})_y=2u_xu_{xy}-2u_{xyx}=2e^u(u_x-u_x)=0.
$$
\end{example}
\begin{definition}
A hyperbolic one-dimensional PDE $u_{xy}=f(u)$ is called Darboux-integrable if it admits both non-trivial $x$-, $y$-integrals.
\end{definition}
Obviously the Liuoville equation $u_{xy}=e^u$ is Darboux-integrable. Moreover there is a well-known  classical formula found by Liuoville himself for its general solution in terms of two arbitrary functions $\varphi(t), \psi(t)$ of one variable $t$
$$
u(x,y)=\log{\frac{2\varphi'(x)\psi'(y)}{(1-\varphi(x)\psi(y))^2}}.
$$
Technical details of a transition from the formulas for $x,y$-integrals of the Liuoville equation to this explicit expression for $u(x,y)$ can be found in  \cite{I, Crowdy}.

Consider sinh-Gordon equation $u_{xy}=\sinh {u}$. It is well-known that it is not Darboux-integrable but it is integrable by inverse scattering problem method (see \cite{ZhSh1, ZhS} for references). In the framework of inverse scattering method 
one is looking for higher symmetries of the non-linear PDE under the study. We will not discuss details and remark only that we are looking now for non-trivial solutions of so-called defining equation
\begin{equation}
\label{def_equation}
DX(f)\phi=f'(u)\phi.
\end{equation}
\begin{example}
A polynomial 
$
\phi_3(u_1,u_2,u_3)=u_3-\frac{1}{2}u_1^3
$
is a solution of the defining equation (\ref{def_equation}) for the sinh-Gordon equation $u_{xy}=\sinh{u}$. It is not difficult to verify that  for a function $u(x,y)$ satisfying the sinh-Gordon equation
$u_{xy}=\sinh{u}$ we have
$$
(u_{xxx}-\frac{1}{2}u_x^3)_{xy}=\cosh{u}(u_{xxx}-\frac{1}{2}u_x^3).
$$
\end{example}
A method was developed in \cite{ZMHSh} that, with the help of operators from the characteristic Lie algebra $\chi(\sinh)$, to obtain all higher symmetries of the sinh-Gordon equation. 

We finish this section with one simple technical lemma, which we will need in the aequel.
\begin{lemma}[\cite{ZMHSh}]
\label{lemma1}
Let $X$ be a differential operator
$$
X=\sum_{i=1}^{+\infty}P_i\frac{\partial}{\partial u_i}, \; P_i=P_i(u,u_1,\dots, u_n,\dots),
$$
 such that 
$
[X,D]=0.
$
Then  $X=0$.
\end{lemma}
\begin{proof}
The proof from \cite{ZMHSh} is quite elementary and we present it here.
$$
[D,X]=\sum_{i=1}^{+\infty}D(P_i)\frac{\partial}{\partial u_i}- P_1\frac{\partial}{\partial u}-\sum_{i=1}^{+\infty}P_{i+1}\frac{\partial}{\partial u_i}.
$$
It follows that if $[X,D]=0$ then
$$
P_1\equiv0, D(P_i)=P_{i+1}, i=1,2,\dots,n,\dots
$$
It means that all polynomials $P_i$ have to vanish, i.e.
$
P_i\equiv 0, \forall i \ge 1.
$
\end{proof}
\begin{corollary}
\begin{equation}
\label{D[]}
[D,X_1]=\sum_{i=1}^{+\infty}D(D^{i-1}(f))\frac{\partial}{\partial u_i}- f\frac{\partial}{\partial u}-\sum_{i=1}^{+\infty}D^{i}(f)\frac{\partial}{\partial u_i}=-fX_0.
\end{equation}
\end{corollary}

\section{Narrow positively graded Lie algebras and loop algebras}

\begin{definition}
A Lie algebra ${\mathfrak g}$ is called ${\mathbb N}$-graded (positively graded) if there is a decomposition of ${\mathfrak g}$ into
a direct sum of linear subspaces 
$$
{\mathfrak g}{=}\bigoplus_{i=1}^{+\infty}{\mathfrak g}_i, \; [{\mathfrak g}_i, {\mathfrak g}_j] \subset  {\mathfrak g}_{i{+}j}, \; {\rm for\; all\;} i,j \in  {\mathbb N}.
$$
\end{definition}

\begin{example}
Let ${\mathfrak g}$ be a finite-dimensional simple Lie algebra over ${\mathbb K}$. 
Then the Lie algebra $L_+({\mathfrak g})=\oplus_{k=1}^{+\infty} {\mathfrak g} \otimes t^k $
with a bracket $[,]_L$ defined 
by 
$$
[g\otimes P(t), h \otimes Q(t)]_{L}=[g,h]\otimes P(t)Q(t),
$$
where $[,]$ is the Lie bracket in  ${\mathfrak g}$ is ${\mathbb N}$-graded and dimensions of all its homogeneous components are equal to $\dim{\mathfrak g}$. $L_+({\mathfrak g})$ can be regarded as the positive part of the corresponding  ${\mathbb Z}$-graded loop algebra $L({\mathfrak g})=\oplus_{k \in \mathbb Z} {\mathfrak g} \otimes t^k$. 
\end{example}

\begin{definition}[\cite{ShZ}]
A ${\mathbb N}$-graded Lie algebra ${\mathfrak g}$ is called of width $d$ if all its homogeneous components
is uniformly bounded by $d \ge 1$.
\begin{equation}
\label{width}
\dim{{\mathfrak g}_i} \le d, \forall i \in {\mathbb N},
\end{equation}
where the constant $d$ is the smallest with the property (\ref{width}).
\end{definition}
Shalev and Zelmanov introduced a notion of {\it narrow} Lie algebra, i.e.  a ${\mathbb N}$-graded Lie algebra ${\mathfrak g}=\oplus_{i \in {\mathbb N}}{\mathfrak g}_i$  of width $d=1$ or $d=2$.
\begin{example}
The Lie algebra $\mathfrak{m}_0$ is defined
by its infinite basis $e_1, e_2, \dots, e_n, \dots $
with the commutation relations:
$$ \label{m_0}
[e_1,e_i]=e_{i+1}, \; \forall \; i \ge 2.$$
The remaining brackets among basis elements vanish: $[e_i,e_j]=0$ if $i,j \ne 1$.
\end{example}

We will always omit the trivial commutator
relations $[e_i,e_j]=0$ in the definitions of Lie algebras.

\begin{example}
The Lie algebra $\mathfrak{m}_2$ is defined by its
infinite basis $e_1, e_2, \dots, e_n, \dots $
commutating relations:
$$
[e_1, e_i ]=e_{i+1}, \quad \forall \; i \ge 2; \quad \quad
[e_2, e_j ]=e_{j+2}, \quad \forall \; j \ge 3.
$$
\end{example}

\begin{example}
The positive part $W^+$ of the Witt algebra.
It can be also defined by its infinite basis and  commutating relations 
$$
[e_i,e_j]= (j-i)e_{i{+}j}, \; \forall \;i,j \in {\mathbb N}.
$$
\end{example}

These three infinite-dimensional algebras $\mathfrak{m}_0, \mathfrak{m}_2, W^+$  are the
narrowest possible ${\mathbb N}$-graded Lie algebras. They are all generated by two elements $e_1, e_2$ of gradings one and two respectively. 

\begin{example}
The loop algebra ${\mathcal L}({\mathfrak sl}(2,{\mathbb K}))$ and its positive part ${\mathfrak n}_1$.

Consider the loop algebra ${\mathcal L}({\mathfrak sl}(2,{\mathbb K}))={\mathfrak sl}(2, {\mathbb K})\otimes {\mathbb K}[t,t^{-1}]$, where ${\mathbb K}[t,t^{{-}1}]$ is the ring of Laurent polynomials over ${\mathbb K}$. It has a Lie subalgebra of "polynomial loops"
$$
{\mathcal L}({\mathfrak sl}(2,{\mathbb K}))^{\ge 0}={\mathfrak sl}(2, {\mathbb K})\otimes {\mathbb K}[t]
$$
that we will call in the sequel the non-negative part of the loop algebra ${\mathcal L}({\mathfrak sl}(2,{\mathbb K}))$.
Consider an infinite set of polynomial matrices  defined for 
$k \in {\mathbb Z}$ by
\begin{equation}
\label{basis_n_1}
e_{3k{+}1}{=}\frac{1}{2}\begin{pmatrix} 
0 & t^{k}\\
0&0\\
\end{pmatrix}, e_{3k{-}1}{=}
\begin{pmatrix} 
0 & 0\\
t^{k}&0\\
\end{pmatrix}, 
e_{3k}{=}
\frac{1}{2}\begin{pmatrix} 
t^{k} & 0\\
0&{-}t^{k}\\
\end{pmatrix}. 
\end{equation}
Evidently this set of matrices is an infinite basis of the loop algebra 
${\mathcal L}({\mathfrak sl}(2,{\mathbb K}))$. 

The linear span of its half $\langle e_0, e_1, e_2, e_3, \dots, e_n, \dots \rangle$ is an infinite basis  of the non-positive part
${\mathcal L}({\mathfrak sl}(2,{\mathbb K}))^{\ge 0}={\mathfrak sl}(2, {\mathbb K})\otimes {\mathbb K}[t]$. It is ${\mathbb Z}_{\ge 0}$-graded with one-dimensional homogeneous components:
$$
{\mathcal L}({\mathfrak sl}(2,{\mathbb K}))^{\ge 0}=\oplus_{i=0}^{+\infty}\langle e_i \rangle \subset {\mathfrak sl}(2, {\mathbb K})\otimes {\mathbb K}[t],
$$
The structure relations for basic elements $e_i, e_j, i,j \ge 0,$ are given by the rule
\begin{equation}
\label{n_1}
[e_i,e_j]= c_{i,j}e_{i{+}j}, \; c_{i,j}=\left\{ \begin{array}{c} 1, {\rm if} \; j{-}i \equiv 1 \; {\rm mod} \;3;\\
0, {\rm if} \; j{-}i \equiv 0  \; {\rm mod}\; 3;\\
{-}1, {\rm if} \; j{-}i \equiv{-}1 \; {\rm mod} \;3.
\end{array}\right.
\end{equation}
Now consider the positive part ${\mathfrak n}_1$ of the loop algebra ${\mathcal L}({\mathfrak sl}(2,{\mathbb K}))$. It is defined as the linear span $\langle e_1, e_2, e_3, \dots, e_n, \dots \rangle$ and it is a ${\mathbb N}$-graded Lie algebra with one-dimensional homogeneous components:
$$
{\mathfrak n}_1=\oplus_{i=1}^{+\infty}\langle e_i \rangle \subset {\mathfrak sl}(2, {\mathbb K})\otimes {\mathbb K}[t].
$$
The Lie algebra ${\mathfrak n}_1$ is a codimension one ideal in ${\mathcal L}({\mathfrak sl}(2,{\mathbb K}))^{\ge 0}$.
\end{example}

\begin{example} The twisted loop algebra ${\mathcal L}({\mathfrak sl}(3,{\mathbb K}), \mu)$ and its positive part ${\mathfrak n}_2$.

Consider a diagram automorphism $\mu$ of ${\mathfrak sl}(3,{\mathbb K})$ of the second order $\mu^2=1$ \cite{Kac}.
$$
\mu: \begin{pmatrix} a_{11} & a_{12} & a_{13} \\ a_{21} & a_{22} & a_{23} \\
a_{31} & a_{32} & a_{33}\end{pmatrix} \to \begin{pmatrix} {-}a_{33} & a_{23} & {-}a_{13} \\ a_{32} & {-}a_{22} & a_{12} \\
{-}a_{31} & a_{21} & {-}a_{11}\end{pmatrix}
$$
The simple Lie algebra  ${\mathfrak sl}(3,{\mathbb K})$ is decomposed into the sum of eigensubspaces ${\mathfrak g}_0, {\mathfrak g}_1$ of $\mu$ coressponding to eigenvalues $1, {-}1$ respectively
$$
{\mathfrak sl}(3,{\mathbb K})={\mathfrak g}_0\oplus {\mathfrak g}_1, 
 [{\mathfrak g}_0,{\mathfrak g}_0] \subset {\mathfrak g}_0, 
  [{\mathfrak g}_0,{\mathfrak g}_1] \subset {\mathfrak g}_1, 
  [{\mathfrak g}_1,{\mathfrak g}_1] \subset {\mathfrak g}_0.
$$
One can choose a basis $f_{-1}, f_0, f_1, f_2,\dots, f_6$ of ${\mathfrak sl}(3,{\mathbb K})$, such that ${\mathfrak g}_0=\langle f_{-1}, f_0, f_1 \rangle$ and ${\mathfrak g}_1=\langle f_2, f_3, f_4, f_5, f_6 \rangle$
$$
\begin{array}{c}
f_{-1}{=}\begin{pmatrix} 
0 & 0 &0\\
1&0 &0 \\
0& 1 &0\\
\end{pmatrix},
f_{0}{=}\begin{pmatrix} 
1 & 0 &0\\
0&0 &0 \\
0& 0 &{-}1
\end{pmatrix},
f_{1}{=}\begin{pmatrix} 
0 & 1&0\\
0&0 & 1\\
0&0&0\\
\end{pmatrix}, 
f_{2}{=}\begin{pmatrix} 
0 & 0 &0\\
0&0 & 0 \\
1&0&0\\
\end{pmatrix},\\
f_{3}{=}\begin{pmatrix} 
0 & 0 &0\\
1&0 & 0 \\
0&{-}1&0\\
\end{pmatrix},
f_{4}{=}\begin{pmatrix} 
1 & 0 &0\\
0&{-}2 & 0 \\
0&0&1\\
\end{pmatrix},
f_{5}{=}\begin{pmatrix} 
0 & 1 &0\\
0&0 & {-}1 \\
0& 0 &0\\
\end{pmatrix},
f_{6}{=}\begin{pmatrix} 
0 & 0 &1\\
0&0 &0 \\
0& 0 &0\\
\end{pmatrix},
\end{array}
$$
We recall (see \cite{Kac}) that the twisted loop algebra ${\mathcal L}({\mathfrak sl}(3,{\mathbb K}), \mu)$ is a Lie subalgebra of the loop algebra  ${\mathcal L}({\mathfrak sl}(3,{\mathbb K}))$ defined by 
$$
{\mathcal L}({\mathfrak sl}(3,{\mathbb K}), \mu)=\bigoplus_{j\in {\mathbb Z}}{\mathfrak g}_{j ({\rm mod} \; 2)} \otimes t^j. 
$$ 
There is an infinite basis of  ${\mathcal L}({\mathfrak sl}(3,{\mathbb K}), \mu)$  
(see \cite{Kac}, Exercise 8.12). 
\begin{equation}
\label{basis_n_2}
\begin{array}{c}
f_{8k{-}1}{=}f_{-1} \otimes t^{2k}, f_{8k}{=}f_0 \otimes t^{2k},
f_{8k{+}1}{=}f_1\otimes t^{2k}, \\
f_{8k{+}2}{=}f_2 \otimes t^{2k{+}1},
f_{8k{+}3}{=}f_3 \otimes t^{2k{+}1},
f_{8k{+}4}{=}f_4 \otimes t^{2k{+}1},\\
f_{8k{+}5}{=}f_5 \otimes t^{2k{+}1},
f_{8k{+}6}{=}f_6 \otimes t^{2k{+}1}, \; k \in {\mathbb Z}.
\end{array}
\end{equation}

It's easy to calculate the commutators $[f_q,f_l]$ of all these basic elements
\begin{equation}
\label{str_n_2}
[f_q,f_l]= d_{q,l}f_{q{+}l}, \; q,l \in {\mathbb N}.
\end{equation}
where the structure constants $d_{q,l}$ are presented in the Table \ref{structure_const_n_2}.

\begin{table}
\caption{Structure constants for ${\mathfrak n}_2$.}
\label{structure_const_n_2}
\begin{center}
\begin{tabular}{|c|c|c|c|c|c|c|c|c|}
\hline
&&&&&&&&\\[-10pt]
 &$f_{8j}$  &$f_{8j{+}1}$ &$f_{8j{+}2}$&$f_{8j{+}3}$&$f_{8j{+}4}$&$f_{8j{+}5}$&$f_{8j{+}6}$&$f_{8j{+}7}$\\
&&&&&&&&\\[-10pt]
\hline
&&&&&&&&\\[-10pt]
$f_{8i}$ & $0$ &$1$ & ${-}2$ & ${-}1$ &$0$ &$1$ &$2$ &${-}1$\\
&&&&&&&&\\[-10pt]
\hline
&&&&&&&&\\[-10pt]
$f_{8i{+}1}$ & ${-}1$ & $0$ & $1$ & $1$ & ${-}3$ & ${-}2$ & $0$ & $1$\\
&&&&&&&&\\[-10pt]
\hline
&&&&&&&&\\[-10pt]
$f_{8i{+}2}$ & $2$ & ${-}1$ &$0$ &$0$&$0$&$1$&${-}1$&$0$ \\
&&&&&&&&\\[-10pt]
\hline
&&&&&&&&\\[-10pt]
$f_{8i{+}3}$ & $1$ &${-}1$&$0$&$0$&$3$&${-}1$&$1$&${-}2$ \\
&&&&&&&&\\[-10pt]
\hline
&&&&&&&&\\[-10pt]
$f_{8i{+}4}$ & $0$ &$3$&$0$&${-}3$&$0$&$3$&$0$&${-}3$ \\
&&&&&&&&\\[-10pt]
\hline
&&&&&&&&\\[-10pt]
$f_{8i{+}5}$ & ${-}1$ &$2$&${-}1$&$1$&${-}3$&$0$&$0$&${-}1$ \\
&&&&&&&&\\[-10pt]
\hline
&&&&&&&&\\[-10pt]
$f_{8i{+}6}$ & ${-}2$ &$0$&$1$&${-}1$&$0$&$0$&$0$&$1$ \\
&&&&&&&&\\[-10pt]
\hline
&&&&&&&&\\[-10pt]
$f_{8i{+}7}$ & $1$ &${-}1$&$0$&$2$&$3$&$1$&${-}1$&$0$ \\[2pt]
\hline
\end{tabular}
\end{center}
\end{table}
The matrix $(d_{q,l})$ is skew-symmetric, its elements $d_{q,l}$
depend only on the residue generated  by dividing positive integers $q$ and $l$ by $8$. 
Moreover $(d_{q,l})$ satisfy the following relations (see \cite{Kac} for references):
$$
d_{i,j}+d_{q,l}=0, \; {\rm if}\; i{+}q\equiv 0 \; {\rm mod} \; 8, j{+}l \equiv 0 \; {\rm mod} \; 8.
$$
We define the non-negative part ${\mathcal L}({\mathfrak sl}(3,{\mathbb K}), \mu)^{\ge 0}$ of the twisted loop algebra ${\mathcal L}({\mathfrak sl}(3,{\mathbb K}), \mu)$ as
$$
{\mathcal L}({\mathfrak sl}(3,{\mathbb K}), \mu)^{\ge 0}=\bigoplus_{j=0}^{+\infty}{\mathfrak g}_{j ({\rm mod} \; 2)} \otimes t^j. 
$$
It coinsides with 
an infinite-dimensional linear span  $\langle f_0, f_1, f_2, f_3, \dots, f_n, \dots \rangle$.

Now we introduce  the positive part ${\mathfrak n}_2$ of the twisted loop algebra ${\mathcal L}({\mathfrak sl}(3,{\mathbb K}), \mu)$ by setting
$$
{\mathfrak n}_2=\oplus_{i=1}^{+\infty}\langle f_i \rangle=\bigoplus_{j=1}^{+\infty}{\mathfrak g}_{j ({\rm mod} \; 2)} \otimes t^j,
$$
\end{example}
Evidently ${\mathfrak n}_2$ is a ${\mathbb N}$-graded Lie algebra of width one.

Fialowski  classified \cite{Fial} the narrowest ${\mathbb N}$-graded Lie algebras, i.e. ${\mathbb N}$-graded Lie algebras ${\mathfrak g}=\oplus_{i{\in}{\mathbb N}} {\mathfrak g}_i$ with one-dimensional homogeneous components  ${\mathfrak g}_i$ that are  generated by two elements from  ${\mathfrak g}_1$ and ${\mathfrak g}_2$ respectively. Fialowski's classification list contains the Lie algebras ${\mathfrak m}_0$,  ${\mathfrak m}_2$, $W^+$,  ${\mathfrak n}_1$, ${\mathfrak n}_2$ considered above and a special multiparametric family of pairwise non-isomorphic Lie algebras. Later a part of Fialowski's theorem was rediscovered by Shalev and Zelmanov \cite{ShZ}.

\section{Naturally graded pro-nilpotent Lie algebras.}
\begin{definition}
\label{pro-nilpotent}
A  Lie algebra ${\mathfrak g}$ is called pro-nilpotent  if for the ideals  ${\mathfrak g}^i$, ${\mathfrak g}={\mathfrak g}^1$, ${\mathfrak g}^i=[{\mathfrak g},{\mathfrak g}^{i-1}], i\ge 2$,
of its descending central sequence 
we have:
$$
\begin{array}{c}
\cap_{i=1}^{+\infty}{\mathfrak g}^i=\{0\}, \;
\dim{{\mathfrak g}/{\mathfrak g}}^i <+\infty.
\end{array}
$$
\end{definition}
It is clear that a finite-dimensional nilpotent Lie algebra ${\mathfrak g}$ is pro-nilpotent. Moreover,
it follows from the definition \ref{pro-nilpotent} that every quotient ${\mathfrak g}/{\mathfrak g}^i$ of a pro-nilpotent Lie algebra is finite-dimensional nilpotent Lie algebra and there is an inverse spectre of  finite-dimensional nilpotent Lie algebras 
$$
 \dots \stackrel{p_{k{+}2,k{+}1}}{\longrightarrow}{\mathfrak g}/{\mathfrak g}^{k{+}1}\stackrel{p_{k{+}1,k}}{\longrightarrow}{\mathfrak g}/
{\mathfrak g}^{k} \stackrel{p_{k,k{-}1}}{\longrightarrow}
\dots \stackrel{p_{3,2}}{\longrightarrow} {\mathfrak g}/{\mathfrak g}^2 \stackrel{p_{2,1}}{\longrightarrow}{\mathfrak g}/{\mathfrak g}^1,
$$
We denote by $\widehat{\mathfrak g}$ the projective (inverse) limit $\widehat{\mathfrak g}=\varprojlim \limits_{k} {\mathfrak g}/
{\mathfrak g}^k$. We call ${\mathfrak g}$ complete if $\widehat{\mathfrak g}={\mathfrak g}$ (${\mathfrak g}=\widehat{\mathfrak g}$ is an inverse limit of finite-dimensional nilpotent Lie algebras).

For a given pro-nilpotent Lie algebra ${\mathfrak g}$ one can consider a sequence of projections of a pro-nilpotent Lie algebra  ${\mathfrak g}$ to its finite-dimensional quotients:
$$
p_m: {\mathfrak g} \to  {\mathfrak g}/{\mathfrak g}^m, \; m \in {\mathbb N}.
$$
They determine the topology of the inverse limit of finite-dimensional spaces on ${\mathfrak g}$, i.e., smallest topology on  ${\mathfrak g}$ for which all these maps $p_m$ are continuous.
\begin{example}
We have considered three infinite-dimensional ${\mathbb N}$-graded Lie algebras
$$
{\mathfrak m}_0, \;{\mathfrak m}_2, \; W^+.
$$
All of them are pro-nilpotent and not complete.
Their completions $\widehat{\mathfrak m}_0, \widehat{\mathfrak m}_2, \widehat W^+$ are the spaces of formal series $\sum_{k=1}^{+\infty}\alpha_k e_k$ of corresponding basic vectors $e_k, k \in {\mathbb N}$.
\end{example}
\begin{definition}
\label{pro-solvable}
A  Lie algebra ${\mathfrak g}$ is called pro-solvable  if for the ideals  ${\mathfrak g}^{(i)}$, ${\mathfrak g}={\mathfrak g}^{(0)}$, ${\mathfrak g}^{(i)}=[{\mathfrak g}^{(i-1)},{\mathfrak g}^{(i-1)}], i\ge 1$,
of its derived sequence of ideals 
we have:
$$
\begin{array}{c}
\cap_{i=1}^{+\infty}{\mathfrak g}^{(i)}=\{0\}, \;
\dim{{\mathfrak g}/{\mathfrak g}}^{(i)} <+\infty.
\end{array}
$$
\end{definition}
The descending central series $\{\mathfrak{g}^k\}$ of a pro-nilpotent Lie algebra ${\mathfrak g}$ determines
a decreasing filtration 
$$
\mathfrak{g}= \mathfrak{g}^1 \supset \mathfrak{g}^2{=}[\mathfrak{g},\mathfrak{g}] \dots \supset \mathfrak{g}^m \supset \mathfrak{g}^{m+1} \supset \dots,\;  [ \mathfrak{g}^m, \mathfrak{g}^n] \subset  \mathfrak{g}^{m+n}, \; m,n \in {\mathbb N}.
$$
and 
one can consider the associated graded Lie 	algebra ${\rm gr}_C \mathfrak{g}$
$${\rm gr}_C \mathfrak{g}=\oplus_{i=1}^{+\infty}\left({\rm gr}_C \mathfrak{g}\right)_i=\oplus_{i=1}^{+\infty}  \left( \mathfrak{g}^i /\mathfrak{g}^{i{+}1}\right)$$ 
with the  bracket defined on its homogeneous components $\left({\rm gr}_C {\mathfrak g}\right)_i, \left({\rm gr}_C {\mathfrak g}\right)_j$ by 
$$ \left[x{+}\mathfrak{g}^{i{+}1}, y{+}\mathfrak{g}^{j{+}1}\right]=[x,y]{+}\mathfrak{g}^{i{+}j{+}1}, x\in \mathfrak{g}^i , y\in \mathfrak{g}^j.$$

\begin{definition}
A pro-nilpotent Lie algebra $\mathfrak{g}$ is  called naturally graduable if it is isomorphic to its associated graded
${\rm gr}_C \mathfrak{g}$.
\end{definition}
\begin{definition}
A ${\mathbb N}$-grading $\mathfrak{g}=\oplus_{i=1}^{+\infty}{\mathfrak g}_i$ of a naturally graduable pro-nilpotent Lie algebra $\mathfrak{g}$ is  called natural grading if there exist a graded isomorphism 
$$
\varphi: {\rm gr}_C \mathfrak{g} \to \mathfrak{g}, \; \varphi(({\rm gr}_C \mathfrak{g})_i)={\mathfrak g}_i, \; 
i \in {\mathbb N}.
$$
\end{definition}

The Lie algebra $\mathfrak{m}_0$ considered above
is naturally graduable. However its grading of width one considered above is not natural 
$$
\left( {\rm gr}_Cm_0\right)_1=\langle e_1, e_2\rangle, 
\left( {\rm gr}_C m_0\right)_i=\langle e_{i{+}1}\rangle, i \ge 2.
$$

The positive part $W^+$ of the Witt algebra and $\mathfrak{m}_2$ are not naturally graded Lie algebras,
one can easily verify the following isomorphisms:
$$ {\rm gr}_C\mathfrak{m}_2 \cong
{\rm gr}_C W^+ \cong {\rm gr}_C\mathfrak{m}_0 \cong
\mathfrak{m}_0.$$

\begin{definition}
A ${\mathbb N}$-graded Lie algebra ${\mathfrak g}=\oplus_{i=1}^{+\infty}{\mathfrak g}_i$ is naturally graded if and only if
$
[{\mathfrak g}_1,{\mathfrak g}_i]={\mathfrak g}_{i{+}1}, i \in {\mathbb N}.
$
\end{definition}
In particular it means that a naturally graded Lie algebra ${\mathfrak g}=\oplus_{i=1}^{+\infty}{\mathfrak g}_i$ is generated
by  its first homogeneous component ${\mathfrak g}_1$. The equivalence of two different definitions of a naturally graded Lie algebra follows from the basic properties of the descending central series of a Lie algebra.

The notion of naturally graded Lie algebra is the infinite-dimensional generalization of so-called Carnot algebra.
\begin{definition}[\cite{AgrM}]
A finite-dimensional Lie algebra ${\mathfrak g}$ is called Carnot algebra if it admits a ${\mathbb N}$-grading 
${\mathfrak g}=\oplus_{i=1}^n{\mathfrak g}_i$ such that
\begin{equation}
\label{carnot}
[{\mathfrak g}_1,{\mathfrak g}_i]={\mathfrak g}_{i{+}1}, i=1,2,\dots,n-1, \; [{\mathfrak g}_1,{\mathfrak g}_n]=0.
\end{equation}
\end{definition}

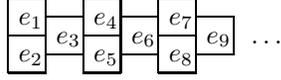
\begin{figure}[tb]
\begin{picture}(112,50)(-33,-5)
\multiput(0,0)(5,0){2}{\line(0,1){10}}
  \multiput(0,0)(0,5){3}{\line(1,0){5}}
  \multiput(10,2.5)(5,0){1}{\line(0,1){5}}
  \multiput(5,2.5)(0,5){2}{\line(1,0){5}}
 \multiput(10,0)(5,0){2}{\line(0,1){10}}
  \multiput(10,0)(0,5){3}{\line(1,0){5}}
\multiput(10,2.5)(5,0){1}{\line(0,1){5}}
  \multiput(15,2.5)(0,5){2}{\line(1,0){5}}
\multiput(20,0)(5,0){2}{\line(0,1){10}}
  \multiput(20,0)(0,5){3}{\line(1,0){5}}
 \multiput(30,2.5)(5,0){1}{\line(0,1){5}}
  \multiput(25,2.5)(0,5){2}{\line(1,0){5}}
\put(32.3, 4){$\dots$}
\put(1.4,1.5){$e_2$}
 \put(1.4, 6.4){$e_1$}
 \put(6.4, 4){$e_3$}
 \put(16.4, 4){$e_6$}
\put(11.4,1.5){$e_5$}
  \put(11.4, 6.4){$e_4$}
\put(21.4,1.5){$e_8$}
  \put(21.4, 6.4){$e_7$}
\put(26.4, 4){$e_9$}
\end{picture}
\caption{The natural grading of ${\mathfrak n}_1$.}
\label{firstfig}
\end{figure}

\begin{proposition}
The Lie algebras ${\mathfrak n}_1$ and ${\mathfrak n}_2$ 
are naturally graded Lie algebras of width two.
\end{proposition}
\begin{proof}
For the proof we will introduce new bases for both algebras. 

In the case ${\mathfrak n}_1$ we define new basic vectors 
$a_{2k{+}1}, b_{2k{+}1}, c_{2k}$ by the rule:
$$
a_{2k{+}1}{=}e_{3k{+}1}, \; b_{2k{+}1}{=}e_{3k{+}2}, \; c_{2k}{=}e_{3k}, \; {\rm for \; all}\; k \in {\mathbb Z}_+.
$$
The structure relations now look as follows
\begin{equation}
\label{n_1^-}
[a_{2k{+}1}, b_{2l{+}1}]{=}c_{2(k{+}l{+}1)},\; [c_{2k}, a_{2l{+}1}]{=}a_{2(k{+}l){+}1}, \;
[c_{2k}, b_{2l{+}1}]{=}{-}b_{2(k{+}l){+}1},
\end{equation}

One can easily verify by recursion that 
$$
C^{2m{+}1}{\mathfrak n}_1={\rm Span}(a_{2m{+}1}, b_{2m{+}1}, c_{2m{+}2},\dots,), \;\; C^{2m}{\mathfrak n}_1={\rm Span}(c_{2m}, a_{2m{+}1}, b_{2m{+}1},\dots,).
$$
Hence the natural grading is defined by 
$$
{\mathfrak n}_1=\oplus_{i{=}0}^{+\infty}{\tilde {\mathfrak n}}_{1,i}, \; {\rm where \;} {\tilde {\mathfrak n}}_{1,2m{+}1}=
 \langle a_{2m{+}1}, b_{2m{+}1} \rangle, {\tilde {\mathfrak n}}_{1,2m}=\langle c_{2m} \rangle
$$
i.e. with one-dimensional even and two-dimensional odd homogeneous components.

In the case ${\mathfrak n}_2$ we define new basic vectors $a_i, b_{6q{+}1}, b_{6q{+}5}$ by
$$
\begin{array}{c}
a_{6q{+}1}{=}e_{8k{+}1},\\
a_{6q{+}2}{=}e_{8k{+}3},\\
a_{6q{+}3}{=}e_{8k{+}4},\\ 
\end{array}
\begin{array}{c}
a_{6q{+}4}{=}e_{8k{+}5},\\
a_{6q{+}5}{=}e_{8k{+}6},\\ 
a_{6q{+}6}{=}e_{8k{+}8},\\
\end{array}
\begin{array}{c}
 b_{6q{+}1}{=}e_{8k{+}2}, \\
b_{6q{+}5}{=}e_{8k{+}7}, 
\end{array}
$$
Then 
$$
{\mathfrak n}_2=\oplus_{i{=}0}^{+\infty}{\tilde {\mathfrak n}}_{2,i}, \; {\tilde {\mathfrak n}}_{2,i}=
 \langle a_{i}, b_{i}\rangle, \;  {\rm if\;} i=6q{+}2\; {\rm or\;}i=6q{+}5, \;   {\tilde {\mathfrak n}}_{2,i}=\langle a_{i} \rangle
\; {\rm in \;other\; cases}. 
$$
The proof  is completely analogous to the previous case and is reduced to the direct calculation of ideals 
$C^k{\mathfrak n}_2$.
\end{proof}
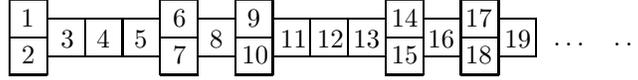
\begin{figure}[tb]
\begin{picture}(112,50)(-13,-5)
  \multiput(0,0)(5,0){2}{\line(0,1){10}}
  \multiput(0,0)(0,5){3}{\line(1,0){5}}
  \multiput(10,2.5)(5,0){2}{\line(0,1){5}}
  \multiput(5,2.5)(0,5){2}{\line(1,0){15}}
 \multiput(20,0)(5,0){2}{\line(0,1){10}}
  \multiput(20,0)(0,5){3}{\line(1,0){5}}
\multiput(30,2.5)(5,0){1}{\line(0,1){5}}
  \multiput(25,2.5)(0,5){2}{\line(1,0){5}}
\multiput(30,0)(5,0){2}{\line(0,1){10}}
  \multiput(30,0)(0,5){3}{\line(1,0){5}}
 \multiput(40,2.5)(5,0){2}{\line(0,1){5}}
  \multiput(35,2.5)(0,5){2}{\line(1,0){15}}
  \multiput(50,0)(5,0){2}{\line(0,1){10}}
  \multiput(50,0)(0,5){3}{\line(1,0){5}}
 \multiput(55,2.5)(5,0){2}{\line(0,1){5}}
  \multiput(55,2,5)(0,5){2}{\line(1,0){5}}
\multiput(60,0)(5,0){2}{\line(0,1){10}}
  \multiput(60,0)(0,5){3}{\line(1,0){5}}
 \multiput(70,2.5)(5,0){1}{\line(0,1){5}}
  \multiput(65,2.5)(0,5){2}{\line(1,0){5}}
\put(72.3, 4){$\dots$}
\put(80.3, 4){$\dots$}
\put(1.6,1.5){$2$}
  \put(1.5, 6.5){$1$}
 \put(6.8, 3.7){$3$}
 \put(11.6, 3.7){$4$}
 \put(16.6, 3.7){$5$}
\put(21.7,1.5){$7$}
  \put(21.7, 6.5){$6$}
\put(26.6, 3.7){$8$}
 \put(30.9,1.5){$10$}
  \put(31.6, 6.5){$9$}
 \put(36, 3.7){$11$}
 \put(40.9, 3.7){$12$}
 \put(45.7, 3.7){$13$}
 \put(50.8,1.5){$15$}
  \put(50.8, 6.5){$14$}
 \put(55.6, 3.7){$16$}
\put(60.6,1.5){$18$}
  \put(60.6, 6.5){$17$}
\put(65.8, 3.7){$19$}
\end{picture}
\caption{The natural grading of ${\mathfrak n}_2$.}
\label{firstfig}
\end{figure}

We define the set of polynomial matrices for positive integers $k=1,2,\dots$: 
$$
u_{2k{-}1}{=}\begin{pmatrix} 0 & t^{2k{-}1} &0\\
{-}t^{2k{-}1} & 0 &0\\
0&0&0
\end{pmatrix},
v_{2k{-}1}^{\pm}{=}\begin{pmatrix} 0 & 0&0\\
 0 & 0& t^{2k{-}1}\\
0& {\mp}t^{2k{-}1} & 0 \\
\end{pmatrix}, 
w_{2k}^{\pm}{=}\begin{pmatrix} 0 & 0& t^{2k}\\
 0 & 0&\\
 {\mp}t^{2k}& 0 & 0 \\
\end{pmatrix}.
$$ 
One can easily verify the commutation relations between them
$$
[u_{2k{-}1}, v_{2l{-}1}^{\pm}]= w_{2(k{+}l){-}2}^{\pm}, \;
[v_{2k{-}1}^{\pm}, w_{2l}^{\pm}]
=\pm u_{2(k{+}l){-}1}, \; 
[w_{2k}^{\pm},u_{2l{+}1}]
= v_{2(k{+}l){-}1}^{\pm}.
$$

The linear span ${\mathfrak n}_1^+{=}\langle u_1, v_1^+, w_2^+, \dots, u_{2k{+}1}, v_{2k{+}1}^+, w_{2k{+}2}^+, \dots \rangle$ is a naturally graded subalgebra in ${\mathfrak so}(3, {\mathbb K})\otimes {\mathbb K}[t]$ and 
${\mathfrak n}_1^-{=}\langle u_1, v_1^-, w_2^-, \dots, u_{2k{+}1}, v_{2k{+}1}^-, w_{2k{+}2}^-, \dots \rangle$ is a naturally graded subalgebra in ${\mathfrak so}(2,1)\otimes {\mathbb K}[t]$ respectively.             
\begin{proposition}[\cite{Mill}]
${\mathfrak n}_1^{\pm}$  are isomorphic
over ${\mathbb C}$ and non-isomorphic over ${\mathbb R}$.
\end{proposition}
 The latter fact is not surprising, given the fact that 
 ${\mathfrak so}(3, {\mathbb R})$ and ${\mathfrak sl}(2, {\mathbb R})$ are the different real forms of ${\mathfrak sl}(2, {\mathbb C})$. 

Let us introduce more examples of naturally graded Lie algebras.

1) Define a Lie algebra $\mathfrak{n}_2^3$ as an one-dimensional central extension of $\mathfrak{n}_2$:
$$\mathfrak{n}_2^3= \mathfrak{n}_2\oplus \langle c \rangle, \; [f_2,f_3]{=}c, \; [c, f_i]{=}0, i \in {\mathbb N}.
$$

2) Let  $S$ be a subset (finite or infinite) of the set of positive odd integers
$$
S=(3 \le 2s_1{+}1 \le 2s_2{+}1 \le  2s_3{+}1 \le \dots \le 2s_n{+}1 \le \dots,)
$$ 
Define a central extension  $\mathfrak{m}_{0}^{S}=\mathfrak{m}_{0}\oplus \langle c_{2s_1{+}1},c_{2s_2{+}1},\dots,c_{2s_n{+}1},\dots \rangle$ of $\mathfrak{m}_0$
$$
\begin{array}{c}
[e_1, e_l]=e_{l+1}, l \ge 2, \; [e_i, e_j]=0, i+j \ne 2s_j{+}1 \in S;\\

[e_k, e_{2s_j{+}1{-}k}]=({-}1)^kc_{2s_j{+}1},k=2, \dots, s_j, \; 2s_j{+}1 \in S, \\ 

[c_{2s_j{+}1}, e_l] = 0, \forall l \in {\mathbb N}, \forall 2s_j{+}1 \in S.
\end{array}
$$ 
\begin{theorem}[\cite{Mill}]
\label{second_osnovn}
Let $\mathfrak{g}=\bigoplus_{i=1}^{+\infty}\mathfrak{g}_i$ be
an infinite-dimensional naturally graded Lie algebra over ${\mathbb R}$ such that
\begin{equation}
\label{3/2}
\dim{\mathfrak{g}_i}+\dim{\mathfrak{g}_{i{+}1}}\le 3, \; i \in {\mathbb N}.
\end{equation}
Then $\mathfrak{g}$ is isomorphic to the one and only one Lie algebra from
the following list:
$$
\mathfrak{n}_1^{\pm}, \mathfrak{n}_2, \mathfrak{n}_2^3, \mathfrak{m}_{0}, \{ \mathfrak{m}_{0}^{S}, S \subset \{3,5,7,\dots,2m{+}1,\dots \}.
$$
\end{theorem}

\section{Kac-Moody algebras $A_1^{(1)}$ and $A_2^{(2)}$}
Let $A$ be a generalized Cartan $(n\times n)$-matrix and ${\mathfrak g}(A)$ be the corresponding Kac-Moody affine algebra (see \cite{Kac} for necessary definitions and details).
By the definition ${\mathfrak g}(A)$ is generated by $3n$ elements $e_i, h_i, f_i, i=1,\dots,n$
satisfying the following relations
\begin{equation}
\label{Kac_Moody}
\begin{split}
[h_i,h_j]=0, [e_i,f_j]=\delta_{ij}h_i,\\ 
[h_i,e_j]=a_{ij}e_j, [h_i,f_j]={-}a_{ij}h_j,\\ 
ad^{1{-}a_{ij}} e_i (e_j)=0, ad^{1{-}a_{ij}} f_i (f_j)=0,\\
i,j=1,\dots,n, 
\end{split}
\end{equation}
where $a_{ij}$ are entries of our generalized Cartan matrix $A$.

 The Kac-Moody affine algebra ${\mathfrak g}(A)$ has the maximal nilpotent subalgebra $N(A) \subset {\mathfrak g}(A)$ and it can be defined by its generators $e_1,e_2,\dots, e_n$ and  the defining set of relations
$$
ad e_i^{{-}a_{ij}{+}1}(e_j)=0, 1 \le i \ne j \le n.
$$ 

The Lie algebra $N(A)$ is ${\mathbb N}\oplus \dots \oplus {\mathbb N}$-graded
\begin{equation}
\label{canon_grad}
N(A)=\oplus_{k_1{>}0,k_2{>}0,{\dots},k_n{>}0}^{{+}\infty} N(A)_{(k_1,k_2,\dots,k_n)},
\end{equation}
where a homogeneous subspace $N(A)_{(k_1,k_2,\dots,k_n)}$ is spanned by all commutator monomials involving precisely 
$k_i$ generators $e_i, i=1,\dots,n$.
\begin{proposition}
 \label{gradings}
$N(A)$ is naturally graded. Natural grading is just the sum  of the components of  canonical grading (\ref{canon_grad})
$$
N(A)=\bigoplus_{N{=}1}^{{+}\infty}N(A)_{(K)}, \; N(A)_{(K)}=\bigoplus_{k_1{+}{\dots}{+}k_n=K} N(A)_{(k_1,\dots,k_n)}
$$
\end{proposition}
\begin{proof}
It follows from the fact that all structure relations of $N(A)$ are defined by homogeneous monomials.
\end{proof}
\begin{definition}
We call the Lie algebra $N(A)$ the positive part of a Kac-Moody algebra ${\mathfrak g}(A)$, where $A$ is the corresponding generalized Cartan matrix.
\end{definition}

It is a classical fact that all Kac-Moody algebras can be realized as affine Lie algebras $\hat {\mathcal L}({\mathfrak g})={\mathcal L}({\mathfrak g})\oplus{\mathbb C}c\oplus{\mathbb C}d$ or twisted affine Lie algebras $\hat {\mathcal L}({\mathfrak g},\mu)={\mathcal L}({\mathfrak g})\oplus{\mathbb C}c\oplus{\mathbb C}d$, double extensions of loop ${\mathcal L}({\mathfrak g})$ and twisted loop algebras ${\mathcal L}({\mathfrak g},\mu)$ respectively (see \cite{Kac}) of complex simple Lie algebras.

We briefly recall the definitions of two affine algebras $A_1^{(1)}$ and $A_2^{(2)}$.

The affine algebra $A_1^{(1)}$ corresponds to $2$ by $2$ generalized Cartan matrix $\begin{pmatrix} 2 & {-}2\\ {-}2 & 2\end{pmatrix}$. It can be realized as a double extension of  of the loop algebra of ${\mathfrak sl}(2,{\mathbb C})$. 
$$
A_1^{(1)}=\hat {\mathcal L}({\mathfrak sl}(2,{\mathbb C}))={\mathcal L}({\mathfrak sl}(2,{\mathbb C}))\oplus {\mathbb C}c \oplus {\mathbb C}d.
$$
Its maximal nilpotent subalgebra $N(A_1^{(1)})$ is isomorphic to the positive part 
${\mathfrak n}_1$ of the loop algebra ${\mathcal L}({\mathfrak sl}(2,{\mathbb C}))$ and it
is generated by two elements $e_1, e_2$ related by
$$
ad^3 e_2 (e_1)=\left[e_2,[e_2, [e_2, e_1]] \right]{=}0, \; ad^{3} e_1(e_2){=}\left[e_1,[e_1,[e_1,e_2]] \right] =0. 
$$
The twisted affine algebra $A_2^{(2)}$ in its turn corresponds  to another generalized $2$ by $2$ Cartan matrix $\begin{pmatrix} 2 & {-}4\\ {-}1 & 2\end{pmatrix}$.
Its maximal nilpotent subalgebra $N(A_2^{(2)})$ is isomorphic to the positive part 
${\mathfrak n}_2$ of the twisted loop algebra ${\mathcal L}({\mathfrak sl}(3,{\mathbb C}),\mu)$ and it is  generated by two elements $e_1, e_2$ related by
$$
ad^2 e_2 (e_1)=\left[e_2, [e_2, e_1] \right]{=}0, \; ad^{5} e_1(e_2){=}\left[e_1,[e_1,[e_1,[e_1,[e_1,e_2]]]] \right] =0. 
$$
Both Lie algebras $A_1^{(1)}$ and $A_2^{(2)}$ are canonically ${\mathbb Z}\oplus {\mathbb Z}$-graded as Kac-Moody algebras, their maximal nilpotent subalgebras ${\mathfrak n}_1=N(A_1^{(1)})$ and ${\mathfrak n}_2=N(A_2^{(2)})$ are ${\mathbb N}\oplus {\mathbb N}$-graded.
$$
N(A_i^{(i)})={\mathfrak n}_i=\oplus_{p,q{=}1}^{{+}\infty} (N(A_i^{(i)}))_{(p,q)}, i=1,2,
$$
where $(N(A_i^{(i)}))_{p,q}, i{=}1,2,$ is the linear span of all commutator monomials involving precisely 
$p$ generators $e_1$ and $q$ generators $e_2$. Generators $e_1, e_2$ have gradings $(1,0)$ and $(0,1)$ respectively. 

How are the gradings of the Lie algebras ${\mathfrak n}_1$ and ${\mathfrak n}_2$, as defined in previous sections, related to the gradings of $N(A_1^{(1)})$ and $N(A_2^{(2)})$ just considered? The connection between the natural grading of ${\mathfrak n}_i$ and the canonical one of $N(A_i^{(i)})$ is already established in the Proposition \ref{gradings}. What about the narrow ${\mathbb N}$-gradings of ${\mathfrak n}_1$ and ${\mathfrak n}_2$?

One can verify that the canonical grading $deg(f_{8m{+}s})$ of a basic element $f_{8m{+}s}$ in ${\mathfrak n}_2$ is defined for ${-}1 \le s \le 6,$ by 
\begin{equation}
\label{A^2_2_grad}
deg(f_{8m{+}s})=\left\{ \begin{array}{l} (4m{+}s, 2m), \; {\rm if}\; s \le 1;\\(4m{+}s{-}2, 2m{+}1), \;{\rm if}\; s \ge 2. \end{array} \right.
\end{equation}
For instance $f_{8m{+}7}$ has the canonical bigrading equal to
$(4m{+}3,2m{+}2)$. In its turn the bigrading of $f_{8m{+}6}$ equals $(4m{+}4,2m{+}1)$. Hence both of them has the natural grading $6m{+}5$.

For canonical bigradings of basic elements $e_i$ of ${\mathfrak n}_1$ we have
$$
deg(e_{3k+1})=(k+1,k), \; deg(e_{3k+2})=(k,k+1), \; deg(e_{3k})=(k,k).
$$

\section{Two-dimensional  integrable hyperbolic systems}

Consider an exponential hyperbolic system
\begin{equation}
\label{exp_Cartan}
u_{x y}^j=e^{\rho_j}, \; \rho_j=a_{j1}u^1+\dots+a_{jn}u^n, \;j=1,\dots, n.
\end{equation}
where $u(x, y)^j, j=1,\dots,n$ are locally analytic functions on variables $x, y$.
For an arbitrary $n$ by $n$ matrix $A$ define vector fields
\begin{equation}
\label{basic_fields}
X_{\alpha}=e^{-\rho_{\alpha}}\sum_{k=1}^{+\infty}D^{k-1}(e^{\rho_{\alpha}})\frac{\partial}{\partial u_k^{\alpha}}=
\sum_{k=1}^{+\infty}B_{k-1}(\rho_{\alpha}^1,\dots,\rho_{\alpha}^{k-1})\frac{\partial}{\partial u_k^{\alpha}},\;  \alpha =1,\dots,n, 
\end{equation}
where $\rho_{\alpha}=a_{\alpha 1}u^1+\dots +a_{\alpha n}u^n$ and we introduced linear functions $\rho_{\alpha}^i, i \ge 1,$ defined by
$$
\rho_{\alpha}^i=a_{\alpha 1}u^1_i+\dots+a_{\alpha n}u^n_i, D(\rho_{\alpha}^i)=\rho_{\alpha}^{i+1}, \; i \ge 1. 
$$
It pas proved in \cite{Leznov} that if $A$ is the Cartan matrix of a semisimple Lie algebra ${\mathfrak g}$ of the rank $n$ then the exponential hyperbolic system (\ref{exp_Cartan}) is integrable. The proof consisted in the explicit construction of a complete solution of equation in
an explicit form which depends on $2n$ arbitrary functions, thus generalizing the one-dimensional case of the classical Liuoville equation $u_{xy}=e^u$. An essential condition in the proof was the nondegeneracy of the Cartan matrix $A$.

Later it was claimed in the preprint \cite{ShYa} that the  main  result in \cite{Leznov} can be generalzed for an arbitrary  generalized Cartan matrix $A$ (possibly degenerate) if we apply the inverse scattering problem method. In the proof \cite{ShYa}, however, there are unclear points.

We consider two-dimensional case $n=2$ that was studied explicitly in  \cite{ShYa, LSmSh}.
$$
\left \{ \begin{array}{c}
u^1_{x y}=e^{(a_{11}u^1+a_{12}u^2)},\\
u^2_{x y}=e^{(a_{21}u^1+a_{22}u^2)},\\
\end{array}\right.,
\; A=\begin{pmatrix}
a_{11} & a_{12}\\
a_{21} & a_{22}
\end{pmatrix}
$$
Characteristic equation $\frac{\partial}{\partial x}w(u_1,u_2,\dots,)=0$ is equivalent to the system 
$$
X_1w=X_2w=0.
$$
where for the  basic fields $X_{\alpha}, \alpha=1,2$, we have the following expansions
$$
X_{\alpha}=\frac{\partial}{\partial u^{\alpha}_1}+(a_{{\alpha}1}u^1_1{+}a_{{\alpha}2}u^2_1)\frac{\partial}{\partial u^{\alpha}_2}+\left((a_{{\alpha}1}u^1_1{+}a_{{\alpha}2}u^2_1)^2{+}(a_{{\alpha}1}u^1_2{+}a_{{\alpha}2}u^2_2)\right)\frac{\partial}{\partial u^{\alpha}_3}+\dots
$$
For instance this system is consistent and it is easy to verify that it has an integral of the second order for arbitrary matrix $A$
$$
w\equiv w^{(2)}(u_1,u_2)=2a_{21}u^1_2+2a_{12}u^2_2-a_{11}a_{21}(u^1_1)^2-2a_{12}a_{21}u_1^1u^2_1-a_{22}a_{12}(u_1^2)^2
$$
In the text of papers \cite{ShYa, LSmSh} we encounter another definition of the characteristic Lie algebra $\chi(A)$ of an exponential hyperbolic system.

\begin{definition}[\cite{LSmSh, ShYa}]
A Lie algebra $\chi(A)$ of vector fields generated by $n$ operators $X_{\alpha}, \alpha=1,\dots,n,$ that are defined by (\ref{basic_fields})
is called characteristic Lie algebra  of  the hyperbolic exponential system (\ref{exp_Cartan}) defined by a matrix $A$.
\end{definition}
\begin{remark}
We do not see operators $\frac{\partial}{\partial u^j}$ among the generators of our algebra. And hence $\chi(A)$ is pro-nilpotent.
\end{remark}

It was proved  in \cite {ShYa, LSmSh} that 

1) for the generalized degenerate Cartan matrix $A=\begin{pmatrix} 2 & {-}2 \\ {-}2 & 2\end{pmatrix}$
the corresponding  characteristic Lie algebra $\chi(A)={\mathcal Lie}_{\mathbb C}(X_1, X_2)$ is isomorphic to the positive part ${\mathfrak n}_1$ of the affine  Kac-Moody algebra $A^{(1)}_1$. The corresponding exponential system is integrable in the framework of the inverse scattering method;

2) for the generalized degenerate Cartan matrix $A=\begin{pmatrix} 2 & {-}4 \\ {-}1 & 2\end{pmatrix}$
the corresponding characteristic  Lie algebra $\chi(A)={\mathcal Lie}_{\mathbb C}(X_1, X_2)$ is isomorphic to the positive part ${\mathfrak n}_2$ of the affine  Kac-Moody algebra $A^{(2)}_2$. Like in the previous case the hyperbolic system is integrable if we apply the inverse scattering problem method.

\begin{remark}
Hyperbolic exponential systems corresponding to nondegenerate Cartan $2\times2$-matrices of semisimple Lie algebras ($A_1{\oplus}A_1, A_2, C_2, G_2$) are Darboux-integrable.
\end{remark}

\section{Growth of  Lie algebras}

In the late sixties Victor Kac  studied simple ${\mathbb Z}$-graded  Lie algebras
${\mathfrak g}=\oplus_{k \in{\mathbb Z}} {\mathfrak g}_k$
of finite growth in the following sense
$$\dim{{\mathfrak g}_k} \le P(|k|), \; k \in {\mathbb Z},$$ for some polynomial $P(t)$.  
We recall that a ${\mathbb Z}$-graded  Lie algebra ${\mathfrak g}=\oplus_{k \in{\mathbb Z}} {\mathfrak g}_k$ is called simple graded if 
it does not contain non-trivial homogeneous ideal $I=\oplus_{k \in {\mathbb Z}}I_k$ where
$I_k=I{\cap} {\mathfrak g}_k$.
Kac \cite{Kac1} proved  that an infinite-dimensional simple ${\mathbb Z}$-graded 
Lie algebra ${\mathfrak g}$ of finite growth that satisfies the following two technical conditions 
\begin{equation}
\label{Kac_condition}
\begin{array}{l}
1) \;\; {\rm {\mathfrak g} \; is \; generated \; by\;  its\; "local\; part"} \; {\mathfrak g}_{-1}{\oplus}{\mathfrak g}_0{\oplus}{\mathfrak g}_1;\\
2) \;\; {\rm  the}\; {\mathfrak g}_0{-}{\rm module} \; {\mathfrak g}_{-1} \; {\rm is \; irreducible}. 
\end{array}
\end{equation}

is isomorphic to one Lie algebra of the following types:

\begin{itemize}
\item
loop algebras  ${\mathcal L}({\mathfrak g})={\mathfrak g}\otimes {\mathbb C}[t,t^{{-}1}]$, where  ${\mathfrak g}$ is finite-dimensional simple Lie algebra and  ${\mathbb C}[t,t^{{-}1}]$ is the ring of Laurent polynomials over complex numbers. Namely there are four infinite series and five exceptional  so-called centerless affine Lie algebras
\cite{Kac}
$$
A_n^{(1)}, B_n^{(1)}, C_n^{(1)}, D_n^{(1)}, E_6^{(1)}, E_7^{(1)}, E_8^{(1)}, F_4^{(1)}, G_2^{(1)}.
$$

\item
twisted loop algebras 
$$
{\mathcal L}({\mathfrak g},\mu)=\bigoplus_{\begin{array}{c}i{\in}{\mathbb Z},i{\equiv}j{\;\rm mod \;n},\\j{=}0,1,{\dots,n{-}1}
\end{array}}{\mathfrak g}_j{\otimes}t^i \subset {\mathfrak g}\otimes {\mathbb C}[t,t^{{-}1}],
$$
where a simple finite-dimensional Lie algebra ${\mathfrak g}=\oplus_{i=0}^{n{-}1} {\mathfrak g}_i$ is graded by the cyclic group ${\mathbb Z}_n$ (eigensubspaces of an automorphism $\mu$ of ${\mathfrak g}$). Here we have two infinite series and two exceptional  centerless twisted affine Lie algebras \cite{Kac}
$$
A_n^{(2)}, D_n^{(2)}, E_6^{(2)}, D_4^{(3)}.
$$

\item
 the Lie algebras $W_n, S_n, K_n, H_n$ of Cartan type,
for instance $W_n$ is the Lie algebra of derivations of the ring of polynomials ${\mathbb C}[x_1,\dots,x_n]$;
\end{itemize}
 Moreover, Kac conjectured
that dropping the condition (\ref{Kac_condition}) would add only
 the Witt algebra $W$ to the classification list. 
\begin{remark}
The Witt algebra $W$ and $W_1$ (with no grading) do not satisfy the first condition from (\ref{Kac_condition}).
\end{remark}

Kac's conjecture was proved in 1990 by 
Mathieu \cite{Mat}.

Suppose that an infinite-dimensional Lie algebra ${\mathfrak g}$ is generated by  its
finite-dimensional subspace $V_1({\mathfrak g})$. For $n>1$ we denote by $V_n({\mathfrak g})$ the ${\mathbb K}$-linear span of all products in elements of
$V_1({\mathfrak g})$ of length at most $n$ with arbitrary arrangements of brackets. We have  an ascending
chain of finite-dimensional subspaces of ${\mathfrak g}$:
$$V_1({\mathfrak g}) \subset V_2({\mathfrak g}) \subset\dots \subset  V_n({\mathfrak g}) \subset \dots, \; \cup_{i=1}^{+\infty} V_i({\mathfrak g}) ={\mathfrak g}.$$ 

The Gelfand-Kirillov dimension of ${\mathfrak g}$ \cite{GeK} is
$$
GK dim {\mathfrak g}= \limsup_{n \to +\infty} \frac{\log{\dim{V_n({\mathfrak g})}}}{\log{n}}.
$$
A finite Gelfand-Kirillov dimension means that there exists a polynomial $P(x)$ such that 
$\dim{V_n({\mathfrak g})} < P(n)$ for all $n>1$. 
In particular if ${\mathfrak g}$ is finite-dimensional then $GK dim {\mathfrak g}=0$.

For a pro-nilpotent Lie algebra ${\mathfrak g}$ the growth function $F_{\mathfrak g}(n)$ can be calculated in terms  of codimensions of ideals of its descending central series 
$$
F(n)_{\mathfrak g}=\dim{V_n({\mathfrak g})}=\dim{({\mathfrak g}/C^{n{+}1}{\mathfrak g})}.
$$ 
Let
${\mathfrak g}=\oplus_{i=1}^{+\infty}{\mathfrak g}_i$ be a naturally graded Lie algebra then 
$$
F_{\mathfrak g}(n)=\dim{V_n({\mathfrak g})}=\sum_{i=1}^n\dim{{\mathfrak g}_i}.
$$

For the Lie algebras 
${\mathfrak m}_0, {\mathfrak m}_2$ and $W^+$  considered above we have 
$$F_{W^+}(n)=F_{{\mathfrak m}_0}(n)=F_{{\mathfrak m}_2}(n)=n{+}1$$ 
and it is the slowest possible growth. 

For an arbitrary  naturally graded Lie algebra ${\mathfrak g}=\oplus_{i=1}^{+\infty}{\mathfrak g}_i$ of width $d$ the function $F_{\mathfrak g}(n)$ grows not faster than $dn$:
$$
F_{\mathfrak g}(n) \le dn.
$$
Obviously all  Lie algebras of finite width have $GK dim {\mathfrak g}=1$.

Consider two growth functions of the Lie algebras ${\mathfrak n}_1$ and ${\mathfrak n}_2$.
$$
\frac{3n}{2} \le F_{{\mathfrak n}_1}(n) \le \frac{3n{+}1}{2},
\frac{4n}{3} \le F_{{\mathfrak n}_2}(n) \le \frac{4n{+}2}{3}, \; \forall n \in {\mathbb N}.
$$
Hence the piecewise linear functions $F_{{\mathfrak n}_1}(n)$ and $F_{{\mathfrak n}_1}(n)$ grow on average at rates of $\frac{3}{2}$ and $\frac{4}{3}$ respectively.
\begin{remark}
We have noticed that the growth function $F_{\mathfrak g}(n)$ of a ${\mathbb N}$-graded Lie algebra ${\mathfrak g}=\oplus_{i=1}^{+\infty}{\mathfrak g}_i$ generally speaking is not completely determined by dimensions $\dim {\mathfrak g}_i$ of its homogeneous components. This is true if and only if the ${\mathbb N}$-grading of is natural ${\mathfrak g}=\oplus_{i=1}^{+\infty}{\mathfrak g}_i$.

The next remark is that there is an  continuum family  of pairwise nonisomorphic linearly growing Lie algebras ${\mathfrak m}_0^S$ indexed by subsets $S \subset \{3,5,7,\dots\}$, while according to the Mathieu theorem \cite{Mat} there is only  a countable number of pairwise nonisomorphic simple ${\mathbb Z}$-graded Lie algebras of finite growth. 
\end{remark}
\begin{lemma}
\label{growth_functions}
Suppose $\tilde {\mathfrak g}$ is generated by its finite-dimensional subspace 
$$V_1(\tilde {\mathfrak g})={\mathfrak g}_0\oplus {\mathfrak g}_1,$$ 
where ${\mathfrak g}_0$ is an abelian subalgebra in $\tilde {\mathfrak g}$, the subspace ${\mathfrak g}_1$ is  invariant under ${\mathfrak g}_0$-action on $\tilde {\mathfrak g}$. Assume also that the ${\mathfrak g}_0$-module ${\mathfrak g}_1$ is diagonalizable and corresponding weights (roots) $\alpha_1, \dots, \alpha_q \in {\mathfrak g}_0^*$ are non-zero.
Define a subalgebra ${\mathfrak g}$ generated by the subspace $V_1({\mathfrak g})={\mathfrak g}_1$. 

Then the growth functions $F_{\mathfrak g}(n)$, $F_{\tilde {\mathfrak g}}(n)$ of the Lie algebras ${\mathfrak g}$ and $\tilde {\mathfrak g}$ are related
$$
F_{\tilde {\mathfrak g}}(n)=F_{\mathfrak g}(n)+\dim{{\mathfrak g}_0}.
$$
Hence  ${\mathfrak g}$ and $\tilde {\mathfrak g}$ have equal Gelfand-Kirillov dimensions
$$
GK dim \tilde {\mathfrak g}=GK dim \mathfrak g.
$$
\end{lemma}
\begin{proof}
For simplicity we consider the case $\dim {\mathfrak g}_0=1$. In addition to everything else, this is the case we will need for applications. However the general case is proved in a completely analogous way.  First of all we fix a non-trivial $X_0$ in one-dimensional ${\mathfrak g}_0$. 
Choose a basis $X_{1},\dots, X_{q}$ of ${\mathfrak g}_1$ consisting of eigen-vectors of  $ad X_0$ corresponding to eigenvalues $\lambda_1= \alpha_1(X_0),\dots,\lambda_q=\alpha_q(X_0)$ respectively
$$
adX_0(X_{j})=[X_0,X_{j}]=\lambda_jX_{j}, \; j=1,\dots, q.
$$
Let $X_{i_1,\dots,i_m}=X_{i_1} \dots X_{i_m}$ be an element of ${\mathfrak g}$ represented by a $m$-word, where  $X_{i_s} \in \{X_1,\dots, X_q\}, s=1,\dots,m,$ with arbitrary (but fixed) arrangement of brackets. Then
\begin{equation}
\label{GKreduction}
ad X_0(X_{i_1,\dots,i_m})=(\lambda_{i_1}+\dots+\lambda_{i_m})X_{i_1,\dots,i_m}, \;  q=1,\dots,m,
\end{equation}

We will prove (\ref{GKreduction})  by recursion. We start by $m=2$:
$$
ad X_0([X_{i_1},X_{i_2}])=[\lambda_{i_1}X_{i_1},X_{i_2}]+[X_{i_1},\lambda_{i_2}X_{i_2}] =(\lambda_{i_1}+\lambda_{i_2})[X_{i_1},X_{i_2}].
$$

Assume that (\ref{GKreduction}) is valid for commutators of orders less than $m$. We take a $m$-word $X_{i_1,\dots,i_q,i_{q+1},\dots,i_m}$ that can be written as a bracket $[X_{i_1,\dots,i_q},X_{i_{q+1},\dots,i_m}]$ of its subwords  $X_{i_1,\dots,i_q}$ and $X_{i_{q+1},\dots,i_m}$.
$$
\begin{array}{c}
ad X_0(X_{i_1,\dots,i_q,i_{q+1},\dots,i_m})=ad X_0([X_{i_1,\dots,i_q},X_{i_{q+1},\dots,i_m}])=\\=[(\lambda_{i_1}{+}{\dots}{+}\lambda_{i_q})X_{i_1,\dots,i_q},X_{i_{q+1},\dots,i_m}]+[X_{i_1,{\dots},i_q},(\lambda_{i_{q+1}}{+}{\dots}{+}\lambda_{i_m})X_{i_{q+1},{\dots},i_m}]=\\
=(\lambda_{i_{1}}{+}{\dots}{+}\lambda_{i_m})X_{i_{1},{\dots},i_q,i_{q+1},{\dots},i_m}.
\end{array}
$$
Now we consider an arbitrary element of $\tilde {\mathfrak g}$ represented by a $n$-th order commutator
$X_{i_1,\dots,i_n}$ where some subscripts can equal zero, in other words, $X_{i_1,\dots,i_n}$ may contain $X_0$ in a certain number among its own letters. Let $s$ be a total number of occurences of the letter $X_0$ in the word $X_{i_1,\dots,i_n}$, then
it follows from (\ref{GKreduction}) that
$$
X_{i_1,\dots,i_n} \in V_{n-s}({\mathfrak g}).
$$
It means that ${\mathfrak g}$ is of codimension one in $\tilde {\mathfrak g}$.
Hence the growth functions $F_{\tilde {\mathfrak g}}(n)$ and $F_{\mathfrak g}(n)$ are simply related and we have  
$$
V_n(\tilde {\mathfrak g})=\langle X_0 \rangle \oplus V_n({\mathfrak g}), \; F_{\tilde {\mathfrak g}}(n)=F_{\mathfrak g}(n) +1. 
$$
\end{proof}


\section{A bigraded Lie subalgebra ${\rm Diff}({\mathcal F})$  of differential operators}

We introduce a non-negative grading in the ring ${\mathbb K}[u_1,\dots,u_n,\dots]$ of polynomials over infinite number of variables $u_1,\dots,u_n,\dots$. We define it by recursion with respect to the power of polynomials.

1) We define the gradings (weights) ${\rm wt}(u_n)$ of generators $u_n, n \ge 1,$ and unit $1$
by the rule 
$$
{\rm wt}(1)=0,{\rm wt}(u_n)=n, n \in {\mathbb N}.
$$

2) Let $P_1$ and $P_2$ be two homogeneous polynomials of gradings ${\rm wt}(P_1){=}p_1$ and ${\rm wt}(P_2){=}p_2$ respectively.
Then their product $P_1P_2$ is a homogeneous polynomial of grading $p_1p_2$.

3) Let $P_1$ and $P_2$  be two homogeneous polynomials of weight ${\rm wt}(P_1){=}{\rm wt}(P_2){=}p$. Then
their sum $P_1+P_2$ is a homogeneous polynomial of grading $p$.

For instance $wt(u_1^3u_3)=6$ and a Bell polynomial $B_n(u_1,\dots, u_n)$ is a homogeneous polynomial of grading $n$: 
$$
{\rm wt}(B_2(u_1,u_2))={\rm wt}(u_1^2+u_2)=2, 
{\rm wt}(B_3(u_1,u_2,u_3))={\rm wt}(u_1^3+3u_1u_2+u_3)=3.
$$

Now we consider a subalgebra ${\mathcal F} \subset C^{\omega}(\Omega)[u_1,u_2,\dots]$ of quasipolynomials 
$$
Q(u,u_1,\dots,u_n,\dots)=\sum_{i={-}m}^Me^{\alpha_i u}P_i(u_1,\dots, u_{n_i}),
$$
where
$\alpha_i \in {\mathbb Z}$ and $P_i(u_1,\dots, u_{n_i})$ stands for a polynomial of variables $u_1,\dots, u_{n_i}$ taken from the ring ${\mathbb K}[u_1,\dots,u_n,\dots]$. 

The ${\mathbb K}$-algebra ${\mathcal F}$ admits a ${\mathbb Z}_{\ge 0}{\times}{\mathbb Z}$-grading 
$$
{\mathcal F}=\bigoplus_{k \in{\mathbb Z}_{\ge 0}, q \in{\mathbb Z}}{\mathcal F}_{k,q}, \; {\mathcal F}_{k,q}=\{ e^{qu}P(u_1,\dots,u_n), {\rm wt}(P)=k \}.
$$
This bigrading is compatible with the product structure in the ring  ${\mathcal F}$
$$
{\mathcal F}_{k,q} \cdot {\mathcal F}_{l,r} \subset {\mathcal F}_{k+l,q+r}.
$$
We consider the Lie algebra ${\rm Diff}(C^{\omega}(\Omega)[u_1,u_2,\dots])$ of all derivations of the algebra $C^{\omega}(\Omega)[u_1,u_2,\dots]$ and a Lie subalgebra  ${\rm Diff}({\mathcal F}) \subset {\rm Diff}(C^{\omega}(\Omega)[u_1,u_2,\dots])$ of first order differential operators 
$$
X=\sum\limits_{j=1}^{+\infty} Q_j(u,u_1,\dots,u_n,\dots)\frac{\partial}{\partial u_j},
$$
where $Q_j(u,u_1,\dots,u_n,\dots) \in {\mathcal F}$ are quasipolynomials.

The Lie subalgebra  ${\rm Diff}({\mathcal F})$ is  ${\mathbb Z}{\times}{\mathbb Z}$-graded
$$
{\rm Diff}({\mathcal F})=\bigoplus_{m \in{\mathbb Z}, r \in{\mathbb Z}} {\rm Diff}_{m,r}({\mathcal F}),\; [{\rm Diff}_{m,r}({\mathcal F}), {\rm Diff}_{n,q}({\mathcal F})] \subset  {\rm Diff}_{m{+}n,r{+}q}({\mathcal F}),
$$
where a homogeneous subspace 
 ${\rm Diff}_{m,r}({\mathcal F})$  is a linear subspace of first order differential operators 
\begin{equation}
\label{operator_grading}
{\rm Diff}_{m,r}({\mathcal F})=\left\{e^{ru}\sum\limits_{j=1}^{+\infty} P_j(u_1,{\dots},u_{s_j})\frac{\partial}{\partial u_j}, {\rm wt}(P_j)=j+m\right\}, (m,r) \in {\mathbb Z}{\times}{\mathbb Z}.
\end{equation}
\begin{definition}
The grading of  ${\rm Diff}_{m,r}({\mathcal F})$ defined by (\ref{operator_grading}) we will call
the operator bigrading of  ${\rm Diff}({\mathcal F})$.
\end{definition}
\begin{example}
$$
X_1=X(e^{p u})=e^{p u}\sum_{n=1}^{+\infty}B_{n-1}(u_1,\dots,u_{n-1})\frac{\partial}{\partial u_n} \in {\rm Diff}_{p,1}({\mathcal F}). 
$$ 
\end{example}

\begin{remark}
Although $X_0=\frac{\partial}{\partial u} \notin {\rm Diff}{\mathcal F}$ its adjoint 
${\rm ad} X_0$ defines a derivation of ${\rm Diff}{\mathcal F}$
\begin{equation}
\label{adX_0}
{\rm ad} X_0(X)=[X_0,X]=\sum\limits_{j=1}^{+\infty} \frac{\partial Q_j}{\partial u} \frac{\partial}{\partial u_j}, \; X=\sum\limits_{j=1}^{+\infty} Q_j \frac{\partial}{\partial u_j}.
\end{equation}
One can see that a subspace 
$V_p=\oplus_{n\in {\mathbb Z}} {\rm Diff}_{p,n}({\mathcal F})$ is an eigensubspace of ${\rm ad} X_0$  which corresponds to the eigenvalue $\lambda=p$. We have the decomposition of the Lie algebra  ${\rm Diff}{\mathcal F}$ into a direct sum of eigensubspaces of the operator ${\rm ad} X_0$.
$$
 {\rm Diff}{\mathcal F}=\bigoplus_{p \in {\mathbb Z}}V_p=\bigoplus_{p \in {\mathbb Z}}\left( \oplus_{n \in {\mathbb Z}}{\rm Diff}_{p,n}({\mathcal F})\right).
$$
\end{remark}
\begin{definition}
Define a new Lie algebra 
$$\hat {\rm Diff}{\mathcal F}=C^{\omega}(\Omega)X_0 \oplus_{\ltimes}  {\rm Diff}{\mathcal F}$$ as a semidirect sum of the Lie algebra $C^{\omega}(\Omega)X_0=\{g(u)X_0, g(u) \in  C^{\omega}(\Omega)\}$ acting on ${\rm Diff}{\mathcal F}$ by the formula (\ref{adX_0}). It is possible to extend the operator bigrading to the whole algebra $\hat {\rm Diff}{\mathcal F}$ by setting its value on the element $X_0$ equal to $(0,0)$.
\end{definition}

\section{Sinh-Gordon equation.}
\begin{theorem}
\label{sinh-Gordon_theorem}
The characteristic Lie algebra $\chi(\sinh{u})$ of the sinh-Gordon equation 
$$u_{xy}=\sinh{u}$$
is  isomorphic to the non-negative part 
$$
{\mathcal L}({\mathfrak sl}(2,{\mathbb K}))^{\ge0}={\mathfrak sl}(2,{\mathbb K}) {\otimes} {\mathbb K}[t],
$$ 
of the loop algebra  ${\mathcal L}({\mathfrak sl}(2,{\mathbb K}))={\mathfrak sl}(2,{\mathbb K}) {\otimes} {\mathbb K}[t,t^{-1}]$. 

 It is generated by three elements $X_0', X_1', X_2'$ that satisfy the following  relations
\begin{eqnarray}
[X_0', X_1']=X_1', \; [X_0, X_2']={-}X_2',\\
\label{n_1_relations}
\left[X_1',[X_1',[X_1',X_2']]\right]=0, \; \left[X_2',[X_2', [X_2', X_1']] \right]=0.
\end{eqnarray}
 It particular it means that the subalgebra $\chi(\sinh{u})^+$ generated by $X_1'$ and $X_2'$ is a codimension one ideal in $\chi(\sinh{u})$ and it is isomorphic to the (nilpotent) positive part $N(A_1^{(1)})$
of the Kac-Moody algebra $A_1^{(1)}=\hat {\mathcal L}({\mathfrak sl}(2,{\mathbb K}))={\mathcal L}({\mathfrak sl}(2,{\mathbb K}))\oplus {\mathbb K}c \oplus {\mathbb K}d$. 
\end{theorem}
\begin{proof}
We denote 
$$
X_0=\frac{\partial}{\partial u}, \; X_1=\sum_{n=1}^{+\infty}D^{n-1}(\sinh{u})\frac{\partial}{\partial u_n}.
$$ 
The construction of the characteristic Lie algebra has an inductive nature.
We start with the first order differential operators $X_0, X_1$ and then consider the commutators of higher orders with  the participation of generators $X_0, X_1$.

Consider a linear span $\langle X_0, X_1, Y_1\rangle$, where $Y_1=[X_0, X_1]$.
Choose a new basis in $\langle X_0, X_1, Y_1\rangle$
$$
X_0'=X_0, X_1'=X_1+Y_1, X_2'=X_1-Y_1.
$$
It means that
$$
X_1'=\sum_{n=1}^{+\infty}D^{n-1}(e^u)\frac{\partial}{\partial u_n}, \quad
X_2'=-\sum_{n=1}^{+\infty}D^{n-1}(e^{-u})\frac{\partial}{\partial u_n}.
$$
We have 
$$
\begin{array}{c}
X_1'=e^u\sum_{n=1}^{+\infty}B_{n-1}(u_1,\dots,u_{n-1})\frac{\partial}{\partial u_n}, \\
X_2'=-e^{-u}\sum_{n=1}^{+\infty}B_{n-1}({-} u_1,\dots,{-}u_{n-1})\frac{\partial}{\partial u_n}.
\end{array}
$$
The elements $X_1', X_2'$ are of operator bigradings $(1,1), (1,-1)$ respectively.
Obviously
$$
[X_0',X_1']=X_1', \; [X_0',X_2']=-X_2'.
$$
It's easy to calculate the first terms of the commutator $[X_1', X_2']$
$$
X_3'=[X_1', X_2']=2 \left( \frac{\partial}{\partial u_2}+u_1^2\frac{\partial}{\partial u_4}+5u_1u_2\frac{\partial}{\partial u_5}+\dots \right)
$$
The operator $X_3'$ has operator bigrading $(2,0)$ (it means in particular that all its coefficients do not depend on 
variable $u$) and hence
$$
[X_0',X_3']=\left[\frac{\partial}{\partial u},X_3'\right]=0.
$$
Now we consider $X_4'=-[X_1',X_3']$ of operator bigrading $(3,1)$ and we also can write out some of its first terms
$$
X_4'=-[X_1',X_3']=2e^u\left(\frac{\partial}{\partial u_3}+u_1\frac{\partial}{\partial u_4}+(2u_1^2+u_2)\frac{\partial}{\partial u_5}+\dots\right)
$$
Evidently $[X_0',X_4']=X_4'$.

We define an operator $X_5'$ of operator bigrading $(3,-1)$ as
$$
X_5'=[X_2',X_3']=-2e^{-u}\left(\frac{\partial}{\partial u_3}-u_1\frac{\partial}{\partial u_4}+(2u_1^2-u_2)\frac{\partial}{\partial u_5}+\dots\right).
$$
Obviously 
$$
[X_0',X_5']=-X_5'.
$$
Now we need to involve the operator $D$ in our play.  It has operator bigrading $(-1,0)$.
We start with an obvious remark that  $[D,X_0']=0$.
It follows from (\ref{D[]}) that
\begin{equation}
[D,X_1']=-e^uX_0', \; [D,X_2']=e^{-u}X_0'.
\end{equation}
Hence we have
\begin{equation}
\begin{split}
[D,X_3']=\left[D,[X_1',X_2'] \right]=\left[[D,X_1'],X_2' \right]+\left[X_1',[D,X_2'] \right]=\\=-\left[e^uX_0',X_2' \right]+\left[X_1',e^{-u}X_0' \right]=e^uX_2'-e^{-u}X_1';\\
[D,X_4']=-\left[D,[X_1',X_3'] \right]=-\left[[D,X_1'],X_3' \right]-\left[X_1',[D,X_3'] \right]=\\=\left[e^uX_0',X_3' \right]-\left[X_1',e^uX_2'+e^{-u}X_1' \right]=-e^uX_3'.
\end{split}
\end{equation}

\begin{proposition} $[X_1', X_4']=[X_2',X_5']=0.$
\end{proposition}
\begin{proof}
$$
\begin{array}{c}
\left[D,[X_1',X_4'] \right]=\left[[D,X_1'],X_4' \right]+\left[X_1',[D,X_4'] \right]=\\=-\left[e^{u}X_0',X_4' \right]-\left[X_1', e^uX_3'\right]=-e^{u}X_4'+e^{u}X_4'=0.
\end{array}
$$
Also we have 
$$
\begin{array}{c}
[D,X_5']=\left[D,[X_2',X_3']\right]=\left[[D,X_2'],X_3' \right]+\left[X_2',[D,X_3'] \right]=\\
=-e^{-u}\left[X_0',X_3' \right]+\left[X_2', e^uX_2'+e^{-u}X_1'\right]=-e^{-u}X_3'.
\end{array}
$$
This implies
$$
\begin{array}{c}
[D,[X_2',X_5']]=\left[[D,X_2'],X_5' \right]+\left[X_2',[D,X_5'] \right]=\\
=-e^{-u}\left[X_0',X_5' \right]-e^{-u}\left[X_2',X_3' \right]=0.
\end{array}
$$
It follows from Lemma  \ref{lemma1} that both brackets $[X_1',X_4']$ and $[X_2',X_5']$ vanish.
\end{proof}
Now we define recursively 
$$
X_{3k+1}'=-[X_1', X_{3k}'], X_{3k+2}'=[X_2', X_{3k}'], X_{3k+3}'=[X_{1}', X_{3k+2}'], k \ge 1.  
$$
 $X_{3k+1}', X_{3k+2}', X_{3k+3}'$ have bigradings $(2k{+}1,1), (2k{+}1,-1), (2k{+}2,0)$ respectively.
\begin{equation}
[X_0',X_{3k{+}1}']=X_{3k{+}1}', [X_0',X_{3k{+}2}']={-}X_{3k{+}2}', [X_0',X_{3k}']=0.
\end{equation}
\begin{lemma} First order differential operators $X_{3k+1}', X_{3k+2}', X_{3k+3}', k \ge 0,$ are all non-trivial and satisfy the following relations
\begin{equation}
\label{recursion_relations_1}
\begin{split}
[D,X_{3k{+}1}']={-}e^uX_{3k}', [D,X_{3k{+}2}']=e^{{-}u}X_{3k}', \\
[D,X_{3k{+}3}']=-e^{-u}X_{3k{+}1}'+e^{u}X_{3k{+}2}'; 
\end{split}
\end{equation}
\end{lemma}
\begin{proof}
$$
\begin{array}{c}
[D,X_{3k+1}']=-\left[D,[X_1',X_{3k}'] \right]=-\left[[D,X_1'],X_{3k}' \right]-\left[X_1',[D,X_{3k}'] \right]=\\=\left[e^uX_0',X_{3k}' \right]-\left[X_1',{-}e^{-u}X_{3(k-1)+1}'+e^{u}X_{3(k-1)+2}' \right]=-e^{u}X_{3k}', 
\end{array}
$$
Second relation from (\ref{recursion_relations_1}) can be proved completely analogously. The third assertion is verified below
$$
\begin{array}{c}
[D,X_{3k+3}']=\left[D,[X_1',X_{3k+2}'] \right]=\left[[D,X_1'],X_{3k+2}' \right]+\left[X_1',[D,X_{3k+2}'] \right]=\\=-\left[e^uX_0',X_{3k+2}' \right]+\left[X_1',e^{-u}X_{3k}'\right]=e^{u}X_{3k+2}'-e^{-u}X_{3k+1}', 
\end{array}
$$
Non-triviality of $X_{3k+1}', X_{3k+2}', X_{3k+3}', k \ge 0,$ follows from Lemma \ref{lemma1} and (\ref{recursion_relations_1}).
\end{proof}
\begin{lemma}
Differential operators $X_0', X_1', X_2',\dots$ satisfy the following commutation relations
\begin{equation}
\begin{split}
\label{A_1_relations}
[X_{3l{+}1} ,X_{3k{+}1}]=0, \; [X_{3l{+}2}, X_{3k{+}2}]=0,\; 
[X_{3l},X_{3k}]=0, \\
[X_{3l{+}1} ,X_{3k{+}2}]=X_{3(k{+}l)+3}, \; \;  [X_{3l}, X_{3k{+}1}]=X_{3(k{+}l){+}1}, \\
 [X_{3l}, X_{3k{+}2}]=-X_{3(k{+}l){+}2}, \; k,l \ge 0; \\
\end{split}
\end{equation}
\end{lemma}
\begin{proof}
We prove (\ref{A_1_relations}) by recursion on $N=k+l$.  The basis of recursion  is $k+l=1$.
$$
\begin{array}{c}
[X_1', X_4'] =0, [X_2', X_5'] =0, [X_0', X_3'] =0,\\

[X_1', X_2'] =X_3', [X_0', X_1'] =X_1',  [X_0', X_2'] =-X_2', \\

[X_1', X_5'] =X_6', [X_0', X_4'] =X_4',  [X_0', X_5'] =-X_5', \\

[X_4', X_2'] =X_6', [X_3', X_1'] =X_4',  [X_3', X_2'] =-X_5',  
\end{array}
$$
We have already checked out almost all of these formulas. 
It only remains to verify the equality
$[X_4', X_2'] =X_6'$. Indeed
$$
\begin{array}{c}
\left[D, [X_4', X_2']-X_6' \right]=\left[[D,X_4'], X_2'\right]+\left[X_4', [D,X_2']\right]-\left[D,X_6' \right]=\\

=\left[-e^uX_3', X_2'\right]+\left[X_4', e^{-u}X_0'\right]+e^{-u}X_4'-e^{u}X_5'=0.
\end{array}
$$

Suppose that relations (\ref{A_1_relations}) have already been established for $k+l=N$, we now prove them for $k+l=N+1$.
$$
\begin{array}{c}
\left[D,[X_{3l+1}',X_{3k+1}']\right]=\left[[D,X_{3l+1}'],X_{3k+1}' \right]+\left[X_{3l+1}',[D,X_{3k+1}'] \right]=\\
=-e^u\left[X_{3l}',X_{3k+1}' \right]-e^u\left[X_{3l+1}', X_{3k}'\right]=0.
\end{array}
$$
Thereby it follows from Lemma \ref{lemma1} that $[X_{3l+1}',X_{3k+1}']{=}0$. The relations  $[X_{3l{+}2}, X_{3k{+}2}]{=}0$ and 
$[X_{3l},X_{3k}]{=}0$ can be verified absolutely analogously to the previous case. 

Now we turn to the second group of relations (\ref{A_1_relations})
$$
\begin{array}{c}
\left[D,[X_{3l{+}1}',X_{3k{+}2}']-X_{3(k{+}l){+}3}'\right]=\\=\left[[D,X_{3l{+}1}'], X_{3k{+}2}'\right]{+}\left[X_{3l{+}1}', [D,X_{3k{+}2}']\right]{-}\left[D,X_{3(k{+}l){+}3}' \right]=\\

=\left[-e^uX_{3l}',X_{3k{+}2}'\right]+\left[X_{3l{+}1}', e^{-u}X_{3k}'\right]+e^{-u}X_{3(k{+}l){+}1}'-e^{u}X_{3(k{+}l){+}2}'=0.
\end{array}
$$
We leave to the reader in the form of an exercise the proof of the relation $[X_{3l}, X_{3k{+}1}]=X_{3(k{+}l){+}1}$. We finish the proof of our Lemma by verifying the last equality in (\ref{A_1_relations}).
$$
\begin{array}{c}
\left[D, [X_{3l}', X_{3k{+}2}']+X_{3(k{+}l)+2}'\right]=\\=\left[[D,X_{3l}'], X_{3k{+}2}'\right]{+}\left[X_{3l}', [D,X_{3k{+}2}']\right]+\left[D,X_{3(k{+}l){+}2}' \right]=\\

=\left[e^uX_{3(l{-}1)+2}'-e^{-u}X_{3l-2}',X_{3k{+}2}'\right] +e^{-u}\left[X_{3l}',X_{3k}'\right]+e^{-u}X_{3(k{+}l)}'=0.
\end{array}
$$
\end{proof}

\end{proof}
\begin{corollary}
1) The characteristic Lie algebra $\chi(\sin{u})$ of the sin-Gordon equation 
$u_{xy}=\sin{u}$
 is isomorphic to the non-negative part 
$${\mathcal L({\mathfrak so}(2,1),{\mathbb K})}_{\ge 0}={\mathfrak so}(2,1)\otimes {\mathbb K}[t]$$
of the loop algebra 
${\mathcal L}({\mathfrak so}(2,1), {\mathbb K})={\mathfrak so}(2,1)\otimes {\mathbb K}[t, t^{{-}1}]$.

2) the loop algebras ${\mathcal L}({\mathfrak so}(2,1),{\mathbb K})$ and  ${\mathcal L}({\mathfrak sl}(2, {\mathbb K}))$ are non-isomorphic over ${\mathbb K}{=}{\mathbb R}$ and 
are isomorphic over ${\mathbb K}{=}{\mathbb C}$.
\end{corollary}
\begin{table}
\caption{Correspondence table of different gradings of $\chi(\sinh{u})$.}
\label{bigraded_struct_n_1}
\begin{tabular}{|c|c|c|c|c|c|c|}
\hline
&&&&&&\\[-10pt]
width one&$X_0'$  &$X_1'$ &$X_2'$&$X_{3k}'$&$X_{3k+1}'$&$X_{3k+2}'$\\
&&&&&&\\[-10pt]
\hline
&&&&&&\\[-10pt]natural & &$1$&$1$&$2k$&$2k{+}1$&$2k{+}1$ \\[2pt]
\hline
&&&&&&\\[-10pt]canonical &$(0,0)$ &$(1,0)$&$(0,1)$&$(k,k)$&$(k{+}1,k)$&$(k,k{+}1)$ \\[2pt]
\hline
&&&&&&\\[-10pt]${\mathbb Z}_{\ge0}{\times}{\mathbb Z}_3$ &$(0,0)$ &$(1,1)$&$(1,{-}1)$&$(k,0)$&$(k,1)$&$(k,{-}1)$ \\[2pt]
\hline
\end{tabular}
\end{table}

\section{Tzitzeica equation}

\begin{theorem}
\label{Tzitzeica_theorem}
The characteristic Lie algebra $\chi(e^u{+}e^{-2u})$ of the Tzitzeica equation 
$$u_{xy}=e^u+e^{-2u}$$
is  isomorphic to the non-negative part 
$$
{\mathcal L}({\mathfrak sl}(3,{\mathbb K}), \mu)^{\ge0}=\bigoplus_{j=0}^{+\infty}{\mathfrak g}_{j ({\rm mod} \; 2)} \otimes t^j, {\mathfrak sl}(3,{\mathbb K})={\mathfrak g}_0\oplus {\mathfrak g}_1, 
 [{\mathfrak g}_{\alpha},{\mathfrak g}_{\beta}] \subset {\mathfrak g}_{\alpha{+}\beta ({\rm mod} \;2)}, 
$$ 
of the twisted loop algebra  ${\mathcal L}({\mathfrak sl}(3,{\mathbb K}), \mu)=\bigoplus_{j \in {\mathbb Z}}{\mathfrak g}_{j ({\rm mod} \; 2)} \otimes t^j$, where $\mu$ is a diagram automorphism of ${\mathfrak sl}(3,{\mathbb K})$, $\mu^2={\rm Id}$, and ${\mathfrak g}_0$,  ${\mathfrak g}_1$ are eigen-spaces of $\mu$ corresponding to eigen-values $1,{-}1$ respectively. In particular ${\mathfrak g}_0$ is a subalgebra in ${\mathfrak sl}(3,{\mathbb K})$ isomorphic to ${\mathfrak so}(3,{\mathbb K})$ \cite{Kac}.

The Lie algebra $\chi(e^u{+}e^{-2u})$ is generated by three elements $Y_0', Y_1', Y_2'$ that satisfy the following  relations
\begin{eqnarray}
[Y_0', Y_1']=Y_1', \; [Y_0, Y_2']={-}2Y_2',\\
\left[Y_1',[Y_1',[Y_1',[Y_1',[Y_1',Y_2']]_{\dots}\right]=0, \; \left[Y_2', [Y_2', Y_1'] \right]=0.
\end{eqnarray}
 It is a pro-solvable infinite-dimensional Lie algebra. Its subalgebra $\chi(e^u{+}e^{-2u})^+$ generated by two elements
$Y_1', Y_2'$ is isomorphic to
 the (nilpotent) positive part $N(A_2^{(2)})$ of the Kac-Moody algebra $A_2^{(2)}=\hat {\mathcal L}({\mathfrak sl}(3,{\mathbb K}), \mu)={\mathcal L}({\mathfrak sl}(3,{\mathbb K}), \mu) \oplus {\mathbb K}c\oplus {\mathbb K}d$.
\end{theorem}
\begin{proof}
Our proof will consist in constructing an infinite basis $Y_1', Y_2', Y_3', Y_4',\dots$ of $\chi(e^u{+}e^{-2u})$ and verifying that basic fields $Y_n, n \ge 1,$ satisfy the commutation relations (\ref{str_n_2}) of ${\mathcal L}({\mathfrak sl}(3,{\mathbb K}), \mu)^{\ge0}$.

 We recall that by definition the characteristic Lie algebra $\chi(e^u{+}e^{-2u})$ is the Lie algebra generated by two operators
$$
X_0=\frac{\partial}{\partial u}, \; X_1=\sum_{n=1}^{+\infty}D^{n-1}(e^{u}{+}e^{-2u})\frac{\partial}{\partial u_n}.
$$
Consider a linear span $\langle X_0, X_1, Y_1\rangle$, where $Y_1=[X_0, X_1]$.
Let introduce a new basis  $Y_0', Y_1', Y_2'$ in $\langle X_0, X_1, Y_1\rangle$, where
$$
Y_0'=X_0, Y_1'=\frac{2}{3}X_1+\frac{1}{3}Y_1, Y_2'=\frac{1}{3}X_1-\frac{1}{3}Y_1.
$$
We recall explicit expressions for $Y_1'$ and $Y_2'$ in terms of Bell polynomials
\begin{equation}
\label{Y_1_2}
\begin{array}{c}
Y_1'=\sum_{n=1}^{+\infty}D^{n-1}(e^u)\frac{\partial}{\partial u_n}=e^u\sum\limits_{n=1}^{+\infty}B_{n-1}(u_1,\dots,u_{n-1})\frac{\partial}{\partial u_n}, \\
Y_2''=\sum_{n=1}^{+\infty}D^{n-1}(e^{-2u})\frac{\partial}{\partial u_n}=e^{-2u}\sum\limits_{n=1}^{+\infty}B_{n-1}({-}2 u_1,\dots,{-}2u_{n-1})\frac{\partial}{\partial u_n}.
\end{array}
\end{equation}
Obviously we have
$$
[Y_0',Y_1']=Y_1', \; [Y_0',Y_2']=-2Y_2'.
$$
It's easy to calculate the first terms of the expansion for $[Y_1', Y_2']$
$$
Y_3'=[Y_1', Y_2']=-3e^{-u} \left( \frac{\partial}{\partial u_2}-2u_1\frac{\partial}{\partial u_3}+(5u_1^2-3u_2)\frac{\partial}{\partial u_4}+\dots \right)
$$
The operator $Y_3'$ has operator bigrading $(2,{-}1)$ and 
$$
[Y_0',Y_3']=\left[\frac{\partial}{\partial u},Y_3'\right]=-Y_3'.
$$
Now we consider $Y_4'=[Y_1',Y_3']$ (it has operator bigrading $(3,0)$) and we can write down the first terms of the expansion of $Y_4'$
$$
Y_4'=[Y_1',Y_3']=9\left(\frac{\partial}{\partial u_3}-2u_1\frac{\partial}{\partial u_4}+(5u_1^2-5u_2)\frac{\partial}{\partial u_5}+\dots\right)
$$
All the coefficients of the differential operator $Y_4'$ do not depend on the variable $u$ and hence
$
[Y_0',Y_4']=0.
$
We define an operator $Y_5'$ of bigrading $(4,1)$ by
$$
Y_5'=-\frac{1}{3}[Y_1',Y_4']=9e^u\left( \frac{\partial}{\partial u_4}-u_1\frac{\partial}{\partial u_5}+(4u_1^2-6u_2)\frac{\partial}{\partial u_6}+\dots\right).
$$
Obviously $[Y_0',Y_5']=Y_5'$.
We recall that  $[D,Y_0']=0$. Then we deduce that
$$
[D,Y_1']=\sum_{i=1}^{+\infty}D(D^{i-1}(e^u))\frac{\partial}{\partial u_i}- e^u\frac{\partial}{\partial u}-\sum_{i=1}^{+\infty}D^{i}(e^u)\frac{\partial}{\partial u_i}=-e^u\frac{\partial}{\partial u}=-e^uY_0'.
$$
Similarly, we conclude that
$[D,Y_2']=-e^{-2u}Y_0'$.
It holds  that
$$
\begin{array}{c}
[D,Y_3']=\left[D,[Y_1',Y_2'] \right]=\left[[D,Y_1'],Y_2' \right]+\left[Y_1',[D,Y_2'] \right]=\\=-\left[e^uY_0',Y_2' \right]-\left[Y_1',e^{-2u}Y_0' \right]=2e^uY_2'+e^{-2u}Y_1'.
\end{array}
$$
\begin{proposition} $[Y_2', Y_3']=[Y_2',Y_4']=0.$
\end{proposition}
\begin{proof}
$$
\begin{array}{c}
\left[D,[Y_2',Y_3'] \right]=\left[[D,Y_2'],Y_3' \right]+\left[Y_2',[D,Y_3'] \right]=\\=-\left[e^{-2u}Y_0',X_3' \right]+\left[Y_2', 2e^uY_2'+e^{-2u}Y_1'\right]=e^{-2u}Y_3'+e^{-2u}\left[Y_2',Y_1'\right]=0.
\end{array}
$$
Consider the second commutator $[Y_2',Y_4']$
$$
\begin{array}{c}
[D,Y_4']=\left[D,[Y_1',Y_3']\right]=\left[[D,Y_1'],Y_3' \right]+\left[Y_1',[D,Y_3'] \right]=\\
=-e^u\left[Y_0',Y_3' \right]+\left[Y_1', 2e^uY_2'+e^{-2u}Y_1'\right]=e^uY_3'+2e^{2u}Y_3'=3e^uY_3'.
\end{array}
$$
Hence it implies that
$$
\begin{array}{c}
[D,[Y_2',Y_4']]=\left[[D,Y_2'],Y_4' \right]+\left[Y_2',[D,Y_4'] \right]
=-e^{-2u}\left[Y_0',Y_4' \right]+3e^{2u}\left[Y_2',Y_3' \right]=0.
\end{array}
$$
It follows from Lemma  \ref{lemma1} that both brackets $[Y_2',Y_4']$ and $[Y_2',Y_3']$ vanish.
\end{proof}
Now  it's the turn of $[D,Y_5']$.
$$
\begin{array}{c}
-3[D,Y_5']=\left[[D,Y_1'],Y_4' \right]+\left[Y_1',[D,Y_4'] \right]
=-e^u\left[Y_0',Y_4' \right]+\left[Y_1', 3e^uY_3'\right]=3e^uY_4'.
\end{array}
$$
We define the sixth element $Y_6'$ of our basis with operator bigrading $(5,2)$ 
$$
Y_6'=-\frac{1}{2}[Y_1',Y_5']=-9e^{2u}\left( \frac{\partial}{\partial u_5}+u_1\frac{\partial}{\partial u_6}+\dots\right).
$$
Obviously 
$
[Y_0', Y_6']=2Y_6'.
$
\begin{proposition}
$
[Y_1', Y_6']=0.
$
\end{proposition}
\begin{proof}
$$
\begin{array}{c}
-2[D,Y_6']=\left[D,[Y_1',Y_5']\right]=\left[[D,Y_1'],Y_5' \right]+\left[Y_1',[D,Y_5'] \right]=\\
=-e^u\left[Y_0',Y_5' \right]- e^u\left[Y_1',Y_4'\right]=2e^uY_5'.
\end{array}
$$
After that we can calculate the commutator $[D,[Y_1',Y_6']]$
$$
\begin{array}{c}
[D,[Y_1',Y_6']]=\left[[D,Y_1'],Y_6' \right]+\left[Y_1',[D,Y_6'] \right]=\\
=-e^{u}\left[Y_0',Y_6' \right]-\left[Y_1',e^{u}Y_5' \right]=-2e^{u}Y_6'-e^{u}\left[Y_1',Y_5' \right]=0.
\end{array}
$$
Hence $[Y_1',Y_6']$ vanishes.
\end{proof}
We define 
$
Y_7'=[Y_2', Y_5'].
$
The operator $Y_7'$ has operator bigrading $(5,-1)$. One can verify that
$$
[Y_0',Y_7']=-Y_7', \; [D, Y_7']=-e^{-2u}Y_5'.
$$
Indeed
$$
\begin{array}{c}
[D, Y_7']=\left[D,[Y_2',Y_5'] \right]=\left[[D,Y_2'],Y_5' \right]+\left[Y_2',[D,Y_5'] \right]=\\
=-e^{-2u}\left[Y_0',Y_5' \right]-\left[Y_2',e^{u}Y_4' \right]=-e^{-2u}Y_5'.
\end{array}
$$
Remark that 
$$
\begin{array}{c}
\left[D,[Y_3',Y_4'] \right]=\left[[D,Y_3'],Y_4' \right]+\left[Y_3',[D,Y_4'] \right]=\\
=\left[e^{-2u}Y_1'+2e^uY_2',Y_4' \right]-\left[Y_2',3e^uY_3' \right]=e^{-2u}\left[Y_1',Y_4' \right]=-3e^{-2u}Y_5'.
\end{array}
$$
It follows from Lemma \ref{lemma1} that 
$[Y_3',Y_4']=3Y_7'$.
We set 
$$
Y_8'=[Y_1', Y_7'].
$$
The operator $Y_8'$ has bigrading $(6,0)$. We need also the following two relations
$$
[Y_0', Y_8']=0, \; [D, Y_8']=2e^{-2u}Y_6'+e^uY_7'.
$$
Let prove them
$$
\begin{array}{c}
[D,Y_8']=\left[D,[Y_1',Y_7'] \right]=\left[[D,Y_1'],Y_7' \right]+\left[Y_1',[D,Y_7'] \right]=\\
=-e^{u}\left[Y_0',Y_7' \right]-\left[Y_1',e^{-2u}Y_5' \right]=2e^{-2u}Y_6'+e^uY_7'.
\end{array}
$$
Besides this
$$
\begin{array}{c}
\left[D,[Y_2',Y_7'] \right]=\left[[D,Y_2'],Y_7' \right]+\left[Y_2',[D,Y_7'] \right]
={-}e^{-2u}\left[Y_0',Y_7' \right]+\left[Y_2',{-}e^{-2u}Y_5' \right]=0.
\end{array}
$$
We sum up the first results of our calculations and collect the obtained relations
\begin{equation}
\label{base_recursion}
\begin{split}
[D,Y_0']=0,
[D,Y_1']={-}e^{u}Y_0',
[D,Y_2']={-}e^{-2u}Y_0', \\
[D,Y_3']=2e^uY_2'+e^{-2u}Y_1',
[D,Y_4']=3e^uY_3',
[D,Y_5']={-}e^uY_4',\\
[D,Y_6']={-}e^uY_5',
[D,Y_7']=-e^{-2u}Y_5',
[D,Y_8']=2e^{-2u}Y_6'+e^uY_7';\\
[Y_2',Y_3']=[Y_2',Y_4']=[Y_2',Y_7']=0.
\end{split}
\end{equation}
It is time to define all the vectors of our infinite basis. We do this with the help of recursive formulas (we recall that  $Y_1', Y_2'$ are defined by (\ref{Y_1_2}))
\begin{equation}
\label{basis_definition}
\begin{split}
Y_{8k+3}'=[Y_{1}', Y_{8k+2}'], Y_{8k+4}'=[Y_{1}',Y_{8k+3}'], Y_{8k+5}'=-\frac{1}{3}[Y_{1}', Y_{8k+4}'],\\
 Y_{8k+6}'=-\frac{1}{2}[Y_1',Y_{8k+5}'], Y_{8k+7}'=[Y_2',Y_{8k+5}'], Y_{8k+8}'=[Y_{1}', Y_{8k+7}'], \\
Y_{8k+9}'=-[Y_1',Y_{8k+8}'], Y_{8k+10}'=\frac{1}{2}[Y_2',Y_{8k+8}'],\quad k \ge 0.
\end{split}
\end{equation}
By induction, it is easy to establish that they are eigenvectors of the operator ${\rm ad} Y_0'$
\begin{equation}
\begin{split}
[Y_0',Y_{8k{+}1}']{=}Y_{8k{+}1}', [Y_0',Y_{8k{+}2}']{=}{-}2Y_{8k{+}2}', [Y_0',Y_{8k{+}3}']{=}{-}Y_{8k{+}3}', \\ [Y_0',Y_{8k{+}4}']{=}0,
 [Y_0',Y_{8k{+}5}']{=}Y_{8k{+}5}',  [Y_0',Y_{8k{+}6}']{=}2Y_{8k{+}6}', \\ [Y_0',Y_{8k{+}7}']{=}{-}Y_{8k{+}7}', [Y_0',Y_{8k+8}']{=}0.
\end{split}
\end{equation}
\begin{lemma} Operators $Y_n', n \ge 1,$ defined by (\ref{basis_definition}) are all non-trivial. More precisely they satisfy the following relations
\begin{equation}
\label{recursion_relations}
\begin{split}
[D,Y_{8k{+}1}']{=}{-}e^uY_{8k}', [D,Y_{8k{+}2}']{=}{-}e^{{-}2u}Y_{8k}', \\ [D,Y_{8k{+}3}']{=}e^{-2u}Y_{8k{+}1}'{+}2e^{u}Y_{8k{+}2}',
 [D,Y_{8k{+}4}']{=}3e^uY_{8k{+}3}', \\ [D,Y_{8k{+}5}']{=}{-}e^{u}Y_{8k{+}4}', [D,Y_{8k{+}6}']{=}{-}e^{u}Y_{8k{+}5}',\\
[D,Y_{8k{+}7}']{=}{-}e^{-2u}Y_{8k{+}5}', [D,Y_{8k{+}8}']{=}e^uY_{8k{+}7}'{+}2e^{-2u}Y_{8k{+}6}';\\
[Y_2',Y_{8k+2}']=[Y_2',Y_{8k+3}']=[Y_2',Y_{8k+4}']=[Y_2',Y_{8k+7}']=0, \; k \ge 0.
\end{split}
\end{equation}
\end{lemma}
\begin{proof}
We prove the lemma and (\ref{recursion_relations}) by induction on $k$. We have allready verified  the case $k=0$ (see (\ref{base_recursion})).  Suppose that the formulas
(\ref{recursion_relations}) are true for all $l \le k-1$, we prove them for $k$. 
$$
\begin{array}{c}
 [D,Y_{8k{+}1}']{=}{-}\left[D,[Y_{1}',Y_{8k}']\right]{=}{-}\left[[D,Y_{1}'],Y_{8k}'\right]{-}
\left[Y_{1}',[D,Y_{8k}']\right]{=}\\
{=}[e^uY_{0}',Y_{8k}']{-}[Y_{1}',e^{u}Y_{8k-1}'+2e^{-2u}Y_{8k-2}']{=}{-}e^uY_{8k}';\\

 [D,Y_{8k{+}2}']{=}\left[D,[Y_2',Y_{8k}']\right]{=}\left[[D,Y_2'],Y_{8k}'\right]{+}
\left[Y_2',[D,Y_{8k}']\right]{=}\\
{=}[{-}e^uY_{0}', Y_{8k}']{+}[Y_2',e^{u}Y_{8k-1}'+2e^{-2u}Y_{8k-2}']{=}[Y_2',2e^{-2u}Y_{8k-2}']={-}e^{-2u}Y_{8k}'.
\end{array}
$$
We skip some evident steps in our calculations and continue
$$
\begin{array}{c}
 [D,Y_{8k{+}3}']
{=}[{-}e^uX_{0}', Y_{8k{+}2}']{+}[Y_{1}',{-}e^{-2u}Y_{8k}']{=}2e^uY_{8k{+}2}'{+}e^{-2u}Y_{8k{+}1}'; \\
\left[D,[Y_{2}',Y_{8k{+}2}']\right]
{=}[{-}e^{-2u}Y_{0}', Y_{8k{+}2}']{+}[Y_{2}',{-}e^{-2u}Y_{8k}']{=}0;\\

 [D,Y_{8k{+}4}']
{=}[{-}e^uY_{0}', Y_{8k{+}3}']{+}[Y_{1}',2e^uY_{8k{+}2}'{+}e^{-2u}Y_{8k{+}1}']{=}3e^uY_{8k{+}3}'.
\end{array}
$$
The relations 
$
\left[D,[Y_{2}',Y_{8k{+}3}']\right]=
\left[D,[Y_{2}',Y_{8k{+}4}']\right]=0.
$
are verified completely analogously.  Next two steps are
$$
\begin{array}{c}
 -3[D,Y_{8k{+}5}']{=}\left[D,[Y_{1}',X_{8k+4}']\right]{=}\left[[D,Y_{1}'],Y_{8k+4}'\right]{+}
\left[X_{1}',[D,Y_{8k+4}']\right]{=}\\
{=}[{-}e^uY_{0}', Y_{8k{+}4}']{+}[X_{1}',3e^uY_{8k{+}3}']{=}3e^uY_{8k{+}4}',\\

 -2[D,Y_{8k{+}6}']{=}\left[D,[Y_{1}',Y_{8k+5}']\right]{=}\left[[D,Y_{1}'],X_{8k+5}'\right]{+}
\left[X_{1}',[D,Y_{8k+5}']\right]{=}\\
{=}[{-}e^uY_{0}', Y_{8k{+}5}']{+}[Y_{1}',{-}e^uY_{8k{+}4}']{=}-e^uY_{8k{+}5}'+3e^uY_{8k{+}5}'=2e^uY_{8k{+}5}'.
\end{array}
$$
We leave the verifying of the following two relations as an exercise to a reader.
$$
\begin{array}{c}
 [D,Y_{8k{+}7}']{=}-e^{-2u}Y_{8k{+}5}', \;
 [D,Y_{8k{+}8}']{=}e^uY_{8k{+}7}'{+}2e^{-2u}Y_{8k{+}6}'.
\end{array}
$$
 We finish the proof of the lemma by
$$
\begin{array}{c}
\left[D,[Y_{2}',Y_{8k{+}7}']\right]=\left[[D,Y_{2}'],X_{8k{+}7}']\right]+\left[Y_{2}',[D,X_{8k{+}7}']\right]=\\
{=}[{-}e^{-2u}Y_{0}', Y_{8k{+}7}']{+}[Y_{2}',{-}e^{-2u}Y_{8k{+}5}']{=}e^{-2u}Y_{8k{+}7}'
-e^{-2u}Y_{8k{+}7}'{=}0.
\end{array}
$$
\end{proof}
\begin{lemma}
The operators $Y_n', n \ge 1,$ satisfy the relations (\ref{str_n_2})
$$
[Y_q',Y_l']=d_{ql}Y'_{q+l}, \; q,l \in {\mathbb N},
$$
where structure constants $d_{ql}{=}{-}d_{lq}$ are taken from the Table \ref{structure_const_n_2}. 
\end{lemma}
\begin{proof}
We are going to apply the formulas (\ref{recursion_relations}) obtained in the previous Lemma.
$$
\begin{array}{c}
 \left[D,[Y_{8q}',Y_{8l}']\right]{=}\left[[D,Y_{8q}'],Y_{8l}'\right]{+}
\left[Y_{8q}',[D,Y_{8l}']\right]{=}\\
{=}[e^uY_{8q{-}1}'+2e^{-2u}Y_{8q{-}2}, Y_{8l}']{+}[Y_{8q}',e^uY_{8l{-}1}'+2e^{-2u}Y_{8l{-}2}]
{=}0.
\end{array}
$$
Hence $[Y_{8q}',Y_{8l}']=0$. Next relation is
$$
\begin{array}{c}
 \left[D,[Y_{8q}',Y_{8l+1}']\right]{=}\left[[D,Y_{8q}'],Y_{8l+1}'\right]{+}
\left[Y_{8q}',[D,Y_{8l+1}']\right]{=}\\
{=}[e^uY_{8q{-}1}'+2e^{-2u}Y_{8q{-}2}, Y_{8l+1}']{+}[Y_{8q}',{-}e^uY_{8l}']
{=}{-}e^uY_{8(q+l)}'.
\end{array}
$$
It follows that $[Y_{8q}',Y_{8l+1}']=Y_{8(q+l)+1}$ because $[D,Y_{8(q+l)+1}]={-}e^uY_{8(q+l)}'$. Then
$$
\begin{array}{c}
 \left[D,[Y_{8q}',Y_{8l+2}']\right]{=}\left[[D,Y_{8q}'],Y_{8l+2}'\right]{+}
\left[Y_{8q}',[D,Y_{8l+2}']\right]{=}\\
{=}[e^uY_{8q{-}1}'+2e^{-2u}Y_{8q{-}2}, Y_{8l+2}']{+}[Y_{8q}',{-}e^{-2u}Y_{8l}']
{=}2e^{-2u}Y_{8(q+l)}'.
\end{array}
$$
Recall that $[D,Y_{8(q+l)+1}']=-e^{-2u}Y_{8(q+l)}'$. Hence $[Y_{8q}',Y_{8l+2}']=-2Y_{8(q+l)+2}'$. We leave the reader, as an exercise, to prove the relations from the first row of Table \ref{structure_const_n_2}.
$$
\begin{array}{c}
[Y_{8q}',Y_{8l+3}']={-}Y_{8(q+l)+3}, [Y_{8q}',Y_{8l+4}']{=}0, [Y_{8q}',Y_{8l+5}']=Y_{8(q+l)+5}, \\

[Y_{8q}',Y_{8l+6}']=2Y_{8(q+l)+6}, [Y_{8q}',Y_{8l+7}']={-}Y_{8(q+l)+7}. 
\end{array}
$$
Now we switch to the second row of the Table \ref{structure_const_n_2}. We start with
$$
\begin{array}{c}
 \left[D,[Y_{8q+1}',X_{8l+1}']\right]{=}\left[[D,Y_{8q+1}'],Y_{8l+1}'\right]{+}
\left[Y_{8q+1}',[D,Y_{8l+1}']\right]{=}\\
{=}[{-}e^uY_{8q}', Y_{8l{+}1}']{+}[Y_{8q+1}',{-}e^{u}Y_{8l}']{=}0,\\

 \left[D,[Y_{8q+1}',Y_{8l+2}']\right]{=}\left[[D,Y_{8q+1}'],Y_{8l+2}'\right]{+}
\left[X_{8q+1}',[D,Y_{8l+2}']\right]{=}\\
{=}[{-}e^uY_{8q}', Y_{8l{+}2}']{+}[Y_{8q+1}',{-}e^{-2u}Y_{8l}']{=}2e^uY_{8(q+l){+}2}'{+}e^{-2u}Y_{8(q+l){+}1}'.
\end{array}
$$
Hence $[Y_{8q+1}',Y_{8l+2}']=Y_{8(q+l)+3}'$ as 
$[D,Y_{8(q+l)+3}']{=}2e^uY_{8(q+l){+}2}'{+}e^{-2u}Y_{8(q+l){+}1}'$. 
$$
\begin{array}{c}
 \left[D,[Y_{8q+1}',Y_{8l+3}']\right]{=}\left[[D,Y_{8q+1}'],Y_{8l+3}'\right]{+}
\left[Y_{8q+1}',[D,Y_{8l+3}']\right]{=}\\
{=}[{-}e^uY_{8q}', Y_{8l{+}3}']{+}[Y_{8q+1}',e^{-2u}Y_{8q+1}'{+}2e^{u}Y_{8l+2}']{=}
3e^{u}Y_{8(q+l){+}3}'
{=}[D,Y_{8(q+l)+4}].
\end{array}
$$
We conclude that  $[Y_{8q+1}',Y_{8l+3}']=Y_{8(q+l)+4}'$. Continuing in the same way and calculating
step by step commutators $[Y_{8q+r}',Y_{8l+s}']$ with $1 \le r \le s \le 7$ we obtain all structure relations (\ref{str_n_2}).
\end{proof}
We define a Lie algebra isomorphism $\varphi: {\chi}(e^u+e^{-2u}) \to \tilde {\mathfrak n}_2$ by setting 
$$\varphi(Y_n')=f_n, n \ge 0.$$
\end{proof}

Now we have to compare different gradings of $\tilde {\mathfrak n}_2$ and compute its growth fuction $F(n)$.

It follows from the proof of the previous theorem that weighted bigrading of ${\rm Diff}{\mathcal F}$ induces
${\mathbb Z}_{\ge 0}{\times}{\mathbb Z}_5$-grading  on $\tilde {\mathfrak n}_2$. The corresponding bigradings 
of basic elements $Y_n'$ are listed in the Table \ref{bigraded_struct_n_2}.

\begin{table}
\caption{Correspondence table of different gradings of $\chi(e^u{+}e^{-2u})$.}
\label{bigraded_struct_n_2}
\begin{tabular}{|c|c|c|c|c|c|c|c|c|}
\hline
&&&&&&&&\\[-10pt]
width $1$&$Y_{8k}'$  &$Y_{8k+1}'$ &$Y_{8k+2}'$&$Y_{8k+3}'$&$Y_{8k+4}'$&$Y_{8k+5}'$&$Y_{8k+6}'$&$Y_{8k+7}'$\\
&&&&&&&&\\[-10pt]
\hline
&&&&&&&&\\[-10pt]natural &$6k$ &$6k{+}1$&$6k{+}1$&$6k{+}2$&$6k{+}3$&$6k{+}4$&$6k{+}5$&$6k{+}5$ \\[2pt]
\hline
&&&&&&&&\\[-10pt]canon. &$\begin{pmatrix}4k\\2k\end{pmatrix}$ &$\begin{pmatrix}4k{+}1\\2k\end{pmatrix}$&$\begin{pmatrix}4k\\2k{+}1\end{pmatrix}$
&$\begin{pmatrix}4k{+}1\\2k{+}1\end{pmatrix}$&$\begin{pmatrix}4k{+}2\\2k{+}1\end{pmatrix}$&
$\begin{pmatrix}4k{+}3\\2k{+}1\end{pmatrix}$&$\begin{pmatrix}4k{+}4\\2k{+}1\end{pmatrix}$
&$\begin{pmatrix}4k{+}3\\2k{+}2\end{pmatrix}$\\[2pt]
\hline
&&&&&&&&\\[-10pt]${\mathbb Z}_{\ge0}{\times}{\mathbb Z}_5$ &$\begin{pmatrix}6k\\0\end{pmatrix}$ &$\begin{pmatrix}6k{+}1\\1\end{pmatrix}$&$\begin{pmatrix}6k{+}1\\{-}2\end{pmatrix}$&
$\begin{pmatrix}6k{+}2\\{-}1\end{pmatrix}$&$\begin{pmatrix}6k{+}3\\0\end{pmatrix}$&
$\begin{pmatrix}6k{+}4\\1\end{pmatrix}$&$\begin{pmatrix}6k{+}5\\2\end{pmatrix}$&
$\begin{pmatrix}6k{+}5\\{-}1\end{pmatrix}$ \\[2pt]
\hline
\end{tabular}
\end{table}

\section{Final remarks}

The characteristic Lie algebra $\chi(\sinh{u})$ of sinh-Gordon equation $u_{xy}=\sinh{u}$ was studied by Murtazina  and Zhiber in \cite{ZM}. An infinite basis of 
$\chi(\sinh{u})$ was constructed there and commutation relations were found. But the very important Lie algebras isomorphism
$$\chi(\sinh{u}) \cong {\mathcal L}({\mathfrak sl}(2, {\mathbb K})), {\mathbb K}={\mathbb R},{\mathbb C},$$ 
was missed there as well as
different gradings of $\chi(\sinh{u})$.

Sakieva examined the characteristic Lie algebra $\chi(e^u{+}e^{-2u})$ of Tzitzeica equation in \cite{Sakieva}. An infinite basis and commutation relations were found it this case also. But again  the very important Lie algebras isomorphism
$$\chi(e^u{+}e^{-2u}) \cong {\mathcal L}({\mathfrak sl}(3, {\mathbb K}), \mu), {\mathbb K}={\mathbb R},{\mathbb C},$$ 
was missed. 

Note also that existence of isomorphisms with non-negative loops ${\mathcal L}({\mathfrak sl}(2, {\mathbb K}))$ and ${\mathcal L}({\mathfrak sl}(3, {\mathbb K}), \mu)$ was missed despite the established slow linear growth of both algebras $\chi(\sinh{u})$ and $\chi(e^u{+}e^{-2u})$ \cite{ZM, Sakieva}.

\end{document}